\documentclass{amsart}

\usepackage{amssymb}

\newtheorem{theorem}{Theorem}[section]
\newtheorem{lemma}[theorem]{Lemma}
\newtheorem{fact}[theorem]{Fact}
\newtheorem{corollary}[theorem]{Corollary}
\newtheorem*{Ramsey's theorem}{Ramsey's Theorem}

\theoremstyle{definition}
\newtheorem{definition}[theorem]{Definition}

\DeclareMathOperator{\height}{\textit{height}}
\DeclareMathOperator{\set}{\textit{set}}

\theoremstyle{remark}
\newtheorem{remark}[theorem]{Remark}
\newtheorem {question}[theorem]{Question}
\numberwithin{equation}{section}

\begin{document}

\title{Cone avoiding closed sets}

\author{Lu Liu}
\address{Department of Mathematics, Central South University,
ChangSha 410083, China}
\curraddr{Department of Mathematics, Central South University, Railway Campus, He Hua cun No.15, Room 619
ChangSha 410083, China}
\email{g.jiayi.liu@gmail.com}
\thanks{We are grateful to Professor Peter Cholak for his help during the preparation of this paper.}


\date{}

\dedicatory{}


\begin{abstract}
We prove that for an arbitrary subtree $T$ of $2^{<\omega}$ with each
element extendable to a path, a given countable class $\mathcal{M}$
closed under disjoint union, and any set $A$, if none of the members
of $\mathcal{M}$ strongly $k$-enumerate $T$ for
any $k$, then there exists an infinite set contained in either $A$ or
$\bar{A}$ such that for every $C\in\mathcal{M}$, $C\oplus G$ also does
not strongly $k$-enumerate $T$. We give
applications of this result, which include: (1) $\mathsf{RT_2^2}$
doesn't imply $\mathsf{WWKL_0}$; (2) [Ambos-Spies et
al.\ \cite{ambos2004comparing}] $\mathsf{DNR}$ is strictly weaker
than $\mathsf{WWKL_0}$; (3) [Kjos-Hanssen \cite{kjos2009infinite}] for
any Martin-L\"{o}f random set $A$ either $A$ or $\bar{A}$ contains an
infinite subset that does not compute any Martin-L\"{o}f random set;
etc. We also discuss further generalizations of this result.
\end{abstract}

\maketitle

\section{Introduction}
\label{sec1}
The interrelationship between topological properties and computability
theoretic properties (usually computational power) of a class is
widely studied in various branches of recursion theory. Here
topological property is a fairly feeble term; for example such
properties could involve
\begin{itemize}
\item all infinite homogeneous sets of a coloring of $[N]^{n}$.
\item all infinite subsets of $A$ or $\bar{A}$.
\item all paths through a tree $T$. For this case, different
combinatorial or topological conditions on the tree yield
different topological conditions on the corresponding class. For
example: $T$ is finitely branching; homogeneous; $[T]$ is of
positive measure, positive Hausdorff measure; etc.
\item all representations of a continuous function: $[0,1]\rightarrow
[0,1]$.
\end{itemize}

This paper focuses on two kinds of classes, $\textit{Part}_k$ and
$[Q]$, where
\begin{multline*}
\textit{Part}_k(A_1,A_2,\ldots,A_{k-1})= \\ \{X\in2^\omega:(\exists i\leq
k-1,X\subseteq A_i\vee X\subseteq
\omega-\bigcup_{j=1}^{k-1}A_j)\wedge |X|=\infty\},
\end{multline*}
and $[Q]$
is a closed set of $2^\omega$. On the positive side, Hirschfeldt
et al.\ \cite{hirschfeldt2008strength} proved that there exists
$A\in\Delta_2^0$ such that $\textit{DNR}\leq_u \textit{Part}_2(A)$, where
$\textit{DNR}=\{f\in\omega^\omega: \varphi_n(n)\!\uparrow\vee\varphi_n(n)\ne
f(n)\}$, and $\leq_u$ denotes the Muchnik reducibility; i.e.,
$\mathcal{C}_2\leq_u\mathcal{C}_1$ means $\forall X\in
\mathcal{C}_1\exists Y\in\mathcal{C}_2$ such that $Y\leq_T X$; while
on the negative side Dzhafarov and Jockusch \cite{Dzhafarov2009} showed
that for any $0<_T C$ and any $A\in 2^\omega$, $\{C\}\nleq_u
\textit{Part}_2(A)$.
Kjos-Hanssen \cite{kjos2009infinite} showed every real of positive
effective Hausdorff dimension computes an infinite
subset of a Martin-L\"{o}f random set, and Greenberg and Miller
\cite{greenbergdiagonally} showed the same for every $f\textit{-DNR}$ (an
$f\textit{-DNR}$ is a $\textit{DNR}$ $g$ such that $(\forall n) g(n)\leq f(n)$)
where $f$ is a
sufficiently ``slow growing'' computable function. Cholak
et al.\ \cite{Cholak2001} showed that for any $A\in\Delta_2^0$
there exists an infinite low$_2$ set $G$ contained in either $A$ or
$\bar{A}$.

We say we can cone avoid $\mathcal{C}_2$ within
$\mathcal{C}_1$ iff $\mathcal{C}_2\nleq_u\mathcal{C}_1$. In this
paper we generalize the result of the author's \cite{Liu2010RT} which
proves the cone avoidance result for arbitrary $\textit{Part}_k$ and
$[Q]=2\textit{-DNR}$
(where $2\textit{-DNR}=\{X\in 2^{\omega}:\forall
n\ \varphi_n(n)\!\downarrow\rightarrow\varphi_n(n)\ne X(n)\}$).

The following is the main theorem in this paper. For a closed set
$[Q]$ (also denoted as $\mathcal{Q}$) of $2^\omega$ let $Q=\{\rho\in
2^{<\omega}: [\rho]^{\preceq}\cap[Q]\neq\emptyset\}$ (where $[\rho]^{\preceq} =
\{X\in 2^\omega: X\supset \rho \}$).
\begin{definition}[Beigel et al.\ \cite{beigela-enumerations}]
\label{def1}
Fix the canonical representation of finite sets, with each finite set
denoted by $D_n$, where $n$ is the (canonical) index of this finite
set. Let $k$ be a positive integer. Let $\{X_i\}_{i\in\omega}$ be an array of non-empty sets.

A \textit{strong $k$-enumeration of
$\{X_i\}_{i\in\omega}$} is a function $h\in \omega^\omega$ such that $|D_{h(n)}|\leq k$
and $X_n\cap D_{h(n)}\neq \emptyset$. A \textit{strong
constant-bound-enumeration} of $\{X_i\}_{i\in\omega}$ is a strong $k$-enumeration for \emph{some}
$k$.

Let $Q$ be a set of strings, say a tree; in the following text we also regard $Q$ as an array of
sets, i.e.\ $Q = \bigcup_{n\in\omega} Q_n$
where $Q_n$ is the set of $n$-length strings of $Q$. A strong $k$-enumeration of
$Q$ means a strong $k$-enumeration
of $\{Q_i\}_{i\in\omega}$.
\end{definition}

Note that if $Q$ is a tree, $W\subseteq \omega$ is infinite, then a strong $k$-enumeration
of $\{Q_{r_n}\}_{r_n\in W}$ can be effectively translated to a strong $k$-enumeration
of $\{Q_i\}_{i\in \omega}$.

\begin{theorem}
\label{th1}
Suppose $\mathcal{M}=\{C^{(1)},C^{(2)},\ldots\}$ is a countable class
of sets, and $\mathcal{Q}=[Q]$ is a closed set (of $2^{\omega}$) such that
\begin{itemize}
\item $\forall C^{(i)},C^{(j)}\in \mathcal{M}, C^{(i)}\oplus
C^{(j)}\in \mathcal{M}$ and
\item $Q$ doesn't admit a strong
constant-bound-enumeration computable in $\mathcal{M}$
(i.e.\ computable in some member of $\mathcal{M}$).

\end{itemize}
Then for any set $A$, there exists an infinite set $G$ such that
\begin{itemize}
\item $G\subseteq A\vee G\subseteq \bar{A}$ and
\item $\forall C\in\mathcal{M}$, $G\oplus C$ also doesn't compute any
strong constant-bound-e\-num\-e\-ra\-tion of $Q$. In particular,
$G\oplus C$ does not compute any member of $[Q]$, since otherwise
 clearly $G\oplus C$
computes a strong 1-enumeration of Q.
\end{itemize}
\end{theorem}

To gain some interest, we first introduce some applications in reverse
mathematics and algorithm randomness theory, which will be derived
from Theorem \ref{th1} in Section \ref{sec5}.
\newline

Reverse mathematics studies the proof theoretic strength of various second order arithmetic
statements. Several statements are so important and fundamental that they serve as level
lines. Many mathematical theorems are found to be equivalent to these statements and
they are unchanged under small perturbations of themselves. The relationships between
these statements and ``other'' statements draw much attention. $\mathsf{WKL}_0$
is one of these statements. $\mathsf{WKL}_0$ states that every infinite binary
tree admits an infinite path. It is well known that as a second order arithmetic
statement, $\mathsf{WKL}_0$ is equivalent to the
statement that for any set $C$ there exists $B\gg C$, where $B\gg C$ means $B$
is of \textrm{PA}-degree relative to $C$.
A good survey of reverse mathematics is \cite{simpson1999subsystems} or
\cite{friedman1975some}, \cite{friedman1976systems}. One of the second order
arithmetic statements close to $\mathsf{WKL}_0$ is $\mathsf{RT}_2^2$.

\begin{definition}
Let $[X]^k$ denote $\{F\subseteq X: |F|=k\}$. A $k$-coloring $f$ is a function $[X]^n\rightarrow \{1,2,\ldots, k\}$. A
set $H\subseteq [X]^k$ is homogeneous for $f$ iff $f$ is constant on $[H]^k$. A stable coloring $f$ is a
2-coloring of $[\mathbb{N}]^2$ such that $(\forall n\in\mathbb{N})(\exists N)(\forall m>N)$
$f(\{m,n\})=f(\{N,n\})$. For a stable coloring $f$, $f_1=\{n\in\mathbb{N}: (\exists N)(\forall
m>N),f(m,n)=1\}$, $f_2=\mathbb{N}-f_1$.
\end{definition}

\begin{Ramsey's theorem}[Ramsey \cite{ramsey1930problem}]
 For any n and k, every k-coloring of $[\mathbb{N}]^n$ admits an infinite homogeneous set.
 \end{Ramsey's theorem}

Let $\mathsf{RT}_k^n$ denote Ramsey's theorem for $k$-colorings of $[\mathbb{N}]^n$ and $\mathsf{SRT}_k^2$ denote Ramsey's theorem
restricted to stable coloring of pairs. It is clear that $\mathsf{RT}_k^2$ implies $\mathsf{SRT}_k^2$.

 Jockusch \cite{jockusch1972ramsey} showed that for $n>2$ $\mathsf{RT}_2^n$ is equivalent to $\mathsf{ACA}_0$, while Seetapun and Slaman
 \cite{seetapun1995strength} showed that $\mathsf{RT}_2^2$ does not imply $\mathsf{ACA}_0$. As to $\mathsf{WKL}_0$, Jockusch
 \cite{jockusch1972ramsey} proved that $\mathsf{WKL}_0$ does not imply $\mathsf{RT}_2^2$. Whether $\mathsf{RT}_2^2$ implies $\mathsf{WKL}_0$ remained open. A more detailed survey of Ramsey's theorem in view of reverse mathematics can be found in Cholak, Jockusch and Slaman
 \cite{Cholak2001}. Say a set $S$ cone avoids a class $\mathcal{M}$ iff $(\forall C\in\mathcal{M})[ C\not\leq_T S]$.

This problem has been a major focus in reverse mathematics in the past twenty years. The first important progress was made by Seetapun and Slaman
\cite{seetapun1995strength}, where they showed that
\begin{theorem}[Seetapun and Slaman \cite{seetapun1995strength}]
For any countable class of sets $\{C_j\}$, $j\in\omega$, such that each $C_i$ is non-computable, any computable
2-coloring of pairs admits an infinite cone avoiding (for $\{C_j\}$) homogeneous set.
\end{theorem}
\noindent Parallel to this result, using Mathias Forcing in a different manner, Dzhafarov and Jockusch \cite{Dzhafarov2009} Lemma 3.2 proved that
\begin{theorem}[Dzhafarov and Jockusch \cite{Dzhafarov2009}]
For any set $A$ and any countable class $\mathcal{M}$, such that each member of $\mathcal{M}$ is non-computable, there exists an
infinite set $G$ contained in either $A$ or its complement such that $G$ is cone avoiding for $\mathcal{M}$.
\end{theorem}
\noindent The main idea is to restrict the computational complexity (computability power) of the homogeneous set as much as possible, with complexity measured by various measurements. Along this line, with simplicity measured by extent of
lowness, Cholak, Jockusch and Slaman \cite{Cholak2001} Theorem 3.1 showed, by a fairly ingenious argument,
\begin{theorem}[Cholak, Jockusch, and Slaman \cite{Cholak2001}]
For any computable coloring of the unordered pairs of natural numbers with finitely many colors, there is an infinite
$low_2$ homogeneous set $X$.
\end{theorem}

The author's \cite{Liu2010RT} proved that $\mathsf{RT_2^2}$ does not imply $\mathsf{WKL_0}$.
Meanwhile, it had also been wondered whether $\mathsf{RT_2^2}$ implies $\mathsf{WWKL_0}$.
Here $\mathsf{WWKL_0}$ is a weaker version of $\mathsf{WKL_0}$ as follows.

\begin{definition}
$\mathsf{WWKL_0}$: $\forall T\subseteq 2^{<\omega},$ if $T$
is a tree s.t.\ $(\exists a>0\exists b\forall n) \dfrac{|\{\rho\in T_n\}|}{2^n}>\dfrac{a}{b}$,
then there exists an infinite path $X\in [T]$.
\end{definition}
Intuitively, $\mathsf{WWKL_0}$ states: for every infinite tree
$T\subseteq2^{<\omega}$, if $\mu([T])>0$ then there is a path through
$T$. It is also considered as a statement in the language of second order arithmetic.
For background on reverse mathematics see
\cite{simpson1999subsystems}, and \cite{Cholak2001}, which focuses on
the proof theoretic strength of Ramsey's Theorem for pairs.

Our major application is that,
\begin{corollary}
\label{cor1}
Over $\mathsf{RCA}_0$, $\mathsf{RT_2^2}$ \text{ does not imply } $\mathsf{ WWKL_0}$.
\end{corollary}

Here $\mathsf{RCA}_0$ is the axiom which says that for any member $X$ of a model of arithmetic, all sets computable in $X$ should also be included in the model. This
axiom corresponds to the wide intuition in mathematics that things
that are derivable, constructible (in some sense) from a given object
that is known to exist should also be considered to exist.

Let $\textit{DNR}=\{f\in\omega^\omega: \forall
e,f(e)\ne\Phi_e(e)\vee\Phi_e(e)\!\uparrow\}$,
$\textit{DNR}^{X}=\{f\in\omega^\omega: \forall
e,f(e)\ne\Phi_e^X(e)\vee\Phi_e^X(e)\!\uparrow\}$, and $g\textit{-DNR}=\{f\in\omega^\omega: f\in \textit{DNR}\wedge (\forall n)\ f(n)\leq g(n)\}$.
$\mathsf{DNR}$ as a second order arithmetic statement says: $\forall X\exists f\in \mathsf{DNR}^{X}$.
Hirschfeldt et al.\ \cite{hirschfeldt2008strength} Theorem 2.3 showed that $\mathsf{SRT_2^2}$ implies $\mathsf{DNR}$; they constructed a $\Delta_2^0$ set $A$ such that any infinite subset of $A$ or $\bar{A}$ computes a $\textit{DNR}$ function. It is also well known that
$\mathsf{WWKL}_0$ implies $\mathsf{DNR}$.
Combining this with Corollary \ref{cor1} yields:
\begin{corollary}[Ambos-Spies et al.\ \cite{ambos2004comparing}]
\label{cor5}
$\mathsf{DNR}$ is implied by $\mathsf{WWKL_0}$ but not versa vice.
\end{corollary}

Actually, Corollary \ref{cor1} yields more. For definitions of basic
notions from algorithmic randomness and dimension used below, see
\cite{downey2008algorithmic} or \cite{nies2009computability}.
Lutz \cite{lutz2000gales}
first studied the effective version of Hausdorff dimension. Later Mayordomo
\cite{mayordomo2002kolmogorov} and Ryabko \cite{ryabko1984coding} gave a
characterization using Kolmogorov complexity. Earlier than them, some results
of Levin, Cai, Hartmanis and Staiger also indicate the deep relationship between
Hausdorff dimension and Kolmogorov complexity.
\begin{definition}
$\dim(A)=\liminf_{n\rightarrow \infty} \dfrac{K(A\upharpoonright n)}{n}$.
\end{definition}
Clearly for any constant $0< d\leq 1$, $T_d=\{\rho\in2^{<\omega}: \dfrac{K_U(\rho)}{|\rho|}\geq d\}$ is a $\Pi_1^0$ set
that does not admit a computable strong constant-bound-enumeration
 (the proof proceeds exactly the same as Lemma \ref{lem1}),
 therefore the induced $\Pi_1^0$ class satisfies the conditions
 given in Theorem \ref{th1}.

 The following corollary related to algorithmic randomness theory
 answers a question of Joe Miller.
\begin{corollary}
\label{cor4}
There exists a $\textit{DNR}$ function that does not compute
\emph{any} binary sequence with positive effective Hausdorff dimension.
\end{corollary}

Another interesting application in algorithmic randomness theory
is the following.
\begin{corollary}[Kjos-Hanssen \cite{kjos2009infinite}]
For every Martin-L\"{o}f random set $A\in 2^\omega$, there exists an infinite set $G\subseteq A\vee G\subseteq \bar{A}$ such that $G$ does not compute any Martin-L\"{o}f random set.
\end{corollary}

In Section \ref{sec3} we illustrate the basic ideas and demonstrate
the vital part of the construction in an informal style for step 1 and
step $s$,  each parallel with the other. The frame of the proof,
``tree-forcing'', together with a review of Mathias forcing, is given
in Subsection \ref{subsec2}. Definitions that simplify statements and
needed notions are given in Section \ref{sec2}, and basic facts
concerned are given in Section \ref{sec3}. In Section \ref{sec4} we
first define some operations used in the construction which are basic
processes in the construction, then in Subsection \ref{subsec6} give
a concrete construction stated via operations introduced in Subsection
\ref{subsec5}, and in Subsection \ref{subsec7} prove that the
construction does provide the desired set. Section \ref{sec5} gives some
applications of this result. Section \ref{sec6} briefly discusses
generalizations of this result and proposes related questions.

\section{Preliminaries}
\label{sec2}
$\Psi$ is sometimes not a single Turing functional but a pair of
Turing functionals, $\langle\Psi_l,\Psi_r\rangle$. $\vec{\Psi}$
denotes an array of pairs of Turing functionals. $\vec{\Psi_i}$
denotes the $i^{\text{th}}$ component. Similarly $\rho$ is sometimes not a
single string but a pair, $\vec{\rho}$ a sequence of pairs,
$\vec{\rho_{i}}=\langle\rho_{il},\rho_{ir}\rangle$. $\Psi^\rho$ means
$\langle\Psi_l^{\rho_l},\Psi_r^{\rho_r}\rangle$,
$\vec{\Psi}^{\vec{\rho}}$ means
$(\langle\Psi_{1l}^{\rho_{1l}},\Psi_{1r}^{\rho_{1r}}\rangle,\langle\Psi_{2l}^{\rho_{2l}},\Psi_{2r}^{\rho_{2r}}\rangle,\ldots)$.

$V$ denotes a clopen set of $2^\omega$; we equate $V$ with
a finite set of strings, i.e.\
$[V]^{\preceq}=\bigcup_{i=1}^n[\rho_i]^{\preceq}$, where
$\rho_1,\rho_2,\ldots, \rho_n$ are mutually incompatible. Write $\height(V)$
for $\max\{|\rho|:\rho\in V\}$. For a closed set $\mathcal{Q}$ (of
$2^\omega$) identify $\bar{\mathcal{Q}}$ with $\{\rho:[\rho]
\subseteq\bar{\mathcal{Q}}\}$ so that $\rho\in\bar{\mathcal{Q}}$
makes sense.

Let $\set(\rho)$ denote $\{i\in\omega:\rho(i)=1\}$, and let
$\rho/\sigma$
($Z/\sigma$) be $\rho$ ($Z$) with the first $|\sigma|$ many values
replaced by $\sigma$.

For a set $X$, view $X$ as an infinite binary string, and let
$\pi_i^n(X)=X_i$ where $X_i\in 2^\omega$ is such that
$X_i(j)=X((j-1)n+i)$, i.e.\ $X=\bigoplus_{i=1}^n X_i$. $\pi$
can be defined on finite strings in the same way.

We say $X$ codes an ordered $k$-partition of $W$ iff
$\bigcup_{i=1}^k\pi_i^k(X)=W$,
and say that $\pi_i^k(X)$ is the
$i^{\text{th}}$ part of this partition. A tree $T\subseteq 2^{<\omega}$ is an
ordered $k$-partition tree of $W$ iff $\forall X\in[T]$ $X$ is an
ordered $k$-partition of $W$. Note that in this paper, ``partition''
does not mean piecewise disjoint partition.

\begin{definition}
\label{def2}
For $i=1,2,\ldots, u$, let $\mathcal{K}_i=\{K_{i,1},K_{i,2},\ldots,
K_{i,m_i}\}$, where each $K_{i,j}$ is a subset of $\{1,2,\ldots, n\}$,
and let $\mathcal{\vec{K}}=(\mathcal{K}_1,\mathcal{K}_2,\ldots,
\mathcal{K}_u)$. We call $\mathcal{\vec{K}}$ a $u$-supporter of
$\{1,2,\ldots, n\}$ iff for every ordered $u$-partition
(not necessarily pairwise disjoint) of $\{1,2,\ldots, n\}$,
namely $P^{(1)}\cup P^{(2)}\cup \cdots \cup P^{(u)}=\{1,2,\ldots, n\}$, there exists some $\mathcal{K}_i$ and some
$K_{i,j}\in\mathcal{K}_i$ such that $K_{i,j}\subseteq P^{(i)}$.

A sequence of $n$ clopen sets $V^{(1)},V^{(2)},\ldots, V^{(n)}$ is
$u$-disperse iff for any ordered $u$-partition (not
necessarily pairwise disjoint) of $\{1,2,\ldots, n\}$, $P^{(1)}\cup
P^{(2)} \cup \cdots \cup P^{(u)}=\{1,2,\ldots, n\}$, there
exists $i\leq u$ such that $\bigcap_{j\in
P^{(i)}}[V^{(j)}]^\preceq=\emptyset$.
\end{definition}

\begin{definition}
For $n$ many ordered $u$-partitions $X^{(1)},\ldots, X^{(n)}$, and for
$\vec{\mathcal{K}}=\{\mathcal{K}_1,\mathcal{K}_2,\ldots,
\mathcal{K}_u\}$, where each $\mathcal{K}_i$ is a finite class of
finite subsets of $\{1,2,\ldots, n\}$, whose members are denoted by $K_{i,j}$, let
\begin{multline*}
\textit{Cross}(X^{(1)},X^{(2)},\ldots, X^{(n)};\mathcal{\vec{K}})= Y=\\
(\bigoplus_{1\leq j\leq
|\mathcal{K}_1|}Y_{K_{1,j}})\oplus(\bigoplus_{1\leq j\leq
|\mathcal{K}_2|} Y_{K_{2,j}})\oplus\cdots
\oplus (\bigoplus_{1\leq j\leq |\mathcal{K}_u|}
Y_{K_{u,j}}),
\end{multline*}
where $Y_{K_{i,j}}=\bigcap_{p\in K_{i,j}\in\mathcal{K}_i}\pi_i^u(X^{(p)})$,
i.e.\ $Y_{K_{i,j}}$ is the intersection of the $i^{\text{th}}$ parts of those
$u$-partitions $X^{(p)}$ for which $p\in K_{i,j}$.
The order of each $Y_{K_{i,j}}$ in $Y$ ``does not matter''; the point is that
given $Y$,$i,j$ and $\mathcal{\vec{K}}$ one could
uniformly (in $Y$,$i,j$ and $\mathcal{\vec{K}}$) compute each $Y_{K_{i,j}}$.

If each
$\mathcal{K}_i$ consists of
all of the $m$-element subsets of $\{1,2,\ldots, n\}$, we also abbreviate
$\textit{Cross}(X^{(1)},X^{(2)},\ldots, X^{(n)};\mathcal{\vec{K}})$ by
$\textit{Cross}(X^{(1)},X^{(2)},\ldots, X^{(n)};m)$.
For $n$ classes $\mathcal{C}^{(1)},\mathcal{C}^{(2)},\ldots,
\mathcal{C}^{(n)}$ of ordered $u$-partitions,
\begin{multline*}
\textit{Cross}(\mathcal{C}^{(1)},\mathcal{C}^{(2)},\ldots,
\mathcal{C}^{(n)};\mathcal{\vec{K}})=\{Y\in 2^\omega:\exists
X^{(i)}\in \mathcal{C}^{(i)} \textrm{ for }1\leq i\leq n, \\
Y=\textit{Cross}(X^{(1)},\ldots, X^{(n)};\mathcal{\vec{K}})\}.
\end{multline*}
Note that the operation $\textit{Cross}$ can be defined on binary
strings in a natural way; therefore if $T^{(1)},\ldots, T^{(n)}$ are
computable trees, then $\textit{Cross}([T^{(1)}],\ldots,
[T^{(n)}];\mathcal{\vec{K}})$ is a $\Pi_1^0$ class. Let
$\textit{Cross}(T^{(1)},\ldots, T^{(n)};\mathcal{\vec{K}})$ denote the
corresponding computable tree.
\end{definition}

Fixing a bijection from $\omega$ to representations of finite subsets
of $2^{<\omega}$, for any $X$, we view each function $\Psi_e^X$ as enumeration,
identifying $\Psi_e(n)$ with the finite subset
of $2^{<\omega}$ it represents, so that
$|\Psi_e(n)|$,$\Psi_e(n)\cap
X,[\Psi_e(n)]^\preceq=\bigcup_{i=1}^{|\Psi_e(n)|}[\rho_i]^{\preceq}$,
etc.\ make sense (if $\Psi_e(n)\!\uparrow$ we view $\Psi_e(n)$ as the
empty set).

\section{Basic ideas and some steps}
\label{sec3}
We will construct a set $G$ with $G\subseteq A\vee G\subseteq
\bar{A}$, satisfying for all $e,j\in\mathbb{N}$:
\medskip

\begin{itemize}
\item $R(e;j)$: $\Psi_e^{C^{(j)}\oplus G}$ is not a strong
  $e$-enumeration of $Q$. I.e.,
  $\Psi_e^{C^{(j)}\oplus G}$ is not total, or $(\exists
  n)(|\Psi_e^{C^{(j)}\oplus G}(n)|>e\vee \Psi_e^{C^{(j)}\oplus
    G}(n)\cap Q=\emptyset)$, or $(\exists l)(\forall n)(\exists\rho\in \Psi_e^{C^{(j)}\oplus G}(n))
    |\rho| < l$.

\item $P_e$: $|G|\geq e$.
\end{itemize}
\medskip

\noindent
\subsection{\textbf{STEP 1}: Satisfying $\mathbf{R(1;j)}$}
\label{subsec1}
\noindent

\textbf{Case i}:
 If we \emph{can} find some $\rho\in 2^{<\omega}$ such that
 $\set(\rho)\subset A\vee \set(\rho)\subset\bar{A}$
 and some $n$ such that $\Psi_1^{C^{(j)}\oplus
   \rho}(n)\!\downarrow\cap Q=\emptyset\vee |\Psi_1^{C^{(j)}\oplus
   \rho}(n)|>1$, then we can satisfy $R(1;j)$ by finitely extending our
   initial segment requirement to $\rho$. In this case we say case i occurs to $R(1;j)$.

\textbf{Case ii}: If there is no such $\rho$, i.e.\ case i doesn't
occur to $R(1;j)$. Then we make $\Psi_1^{C^{(j)}\oplus G}$ non-total;
we can achieve this by \emph{trying} to find three disjoint clopen
sets $[V^{(1)}]^\preceq,[V^{(2)}]^\preceq,[V^{(3)}]^\preceq$ such that
the following classes are non-empty:
\begin{multline*}
[T_{V^{(i)}}]=\{X=X_l\oplus X_r: X_l\cup X_r=\omega \wedge[(\forall Z
  \textrm{ s.t.\ } Z \subseteq X_l\vee Z\subseteq X_r)(\forall n)\\
\Psi_1^{C^{(j)}\oplus Z}(n)\!\downarrow\rightarrow(|\Psi_1^{C^{(j)}\oplus Z} (n)|\leq 1\wedge[\Psi_1^{C^{(j)}\oplus Z} (n)]^\preceq\cap [V^{(i)}]^\preceq\ne \emptyset)]\}.
\end{multline*}

Let
\[
[T_1]=\textit{Cross}([T_{V^{(1)}}],[T_{V^{(2)}}],[T_{V^{(3)}}]; 2),
\]
i.e.\ $\forall X\in [T_1],\ X=Y_1\oplus Y_2\oplus Y_3\oplus Y_4\oplus
Y_5\oplus Y_6$, where
\begin{eqnarray*}
Y_1=\pi_l^2(X')\cap\pi_l^2(X'')& \textrm{for some} &X'\in [T_{V^{(1)}}],X''\in [T_{V^{(2)}}]\\
Y_2=\pi_l^2(X')\cap\pi_l^2(X'')& \textrm{for some} &X'\in [T_{V^{(2)}}],X''\in [T_{V^{(3)}}]\\
Y_3=\pi_l^2(X')\cap\pi_l^2(X'')& \textrm{for some} &X'\in [T_{V^{(3)}}],X''\in [T_{V^{(1)}}]\\
Y_4=\pi_r^2(X')\cap\pi_r^2(X'')& \textrm{for some} &X'\in [T_{V^{(1)}}],X''\in [T_{V^{(2)}}]\\
Y_5=\pi_r^2(X')\cap\pi_r^2(X'')& \textrm{for some} &X'\in [T_{V^{(2)}}],X''\in [T_{V^{(3)}}]\\
Y_6=\pi_r^2(X')\cap\pi_r^2(X'')& \textrm{for some} &X'\in [T_{V^{(3)}}],X''\in [T_{V^{(1)}}].\\
\end{eqnarray*}

Note:
\begin{enumerate}
\item $[T_{V^{(i)}}]$ is a $\Pi_1^{0,C^{(j)}}$ class. For any general
$\Psi$ (instead of $\Psi_1$) and $V$ (instead of $V^{(i)}$) the index of
$[T_{V}]$, which is induced by $\Psi, V$, can
  be uniformly computed from indices for $\Psi$ and $V$.
\item If $[V]^\preceq\supseteq [Q]$ then $T_{V}\ne\emptyset$, as
  otherwise case i occurs: For $X_1=A$ and $X_2=\bar{A}$, if we assume
  that $X_1\oplus
  X_2\notin T_{V}$ then there are a $\rho$ such that
  $\set(\rho)\subseteq A\vee \set(\rho)\subseteq
  \bar{A}$ and an $n\in\mathbb{N}$ with $[\Psi_1^{C^{(j)}\oplus
      \rho}(n)\!\downarrow]^\preceq \cap
      [V]^\preceq=\emptyset\vee|\Psi_1^{C^{(j)}\oplus\rho}(n)|>1$,
      which implies that $[\Psi_1^{C^{(j)}\oplus \rho}(n)]^\preceq\cap
      Q=\emptyset\vee |\Psi_1^{C^{(j)}\oplus\rho}(n)|>1$.
\item For any $i\leq 6, G\subseteq Y_i$,
$\Psi_1^{C^{(j)}\oplus G}$ is not a strong
1-enumeration of $Q$.
To see this, suppose on the contrary that for some $G\subseteq Y_1$,
$\Psi_1^{C^{(j)}\oplus G}$ is a strong
1-enumeration of $Q$. Then there exists
$n$ such that $\forall \rho\in\Psi_1^{C^{(j)}\oplus
  G}(n),|\rho|=n >\max\{\height(V^{(1)}),\height(V^{(2)})\}$. Then $\forall \rho\in
\Psi_1^{C^{(j)}\oplus G}(n)$ either $[\rho]^\preceq\cap[
  V^{(1)}]^\preceq=\emptyset$ or $[\rho]^\preceq\cap[
  V^{(2)}]^\preceq=\emptyset$ (since $[V^{(1)}],[V^{(2)}]$ are
disjoint), say $[\rho]^\preceq\cap[ V^{(1)}]^\preceq=\emptyset$, which
implies $[\Psi_1^{C^{(j)}\oplus G}(n)]\cap
[V^{(1)}]^\preceq=\emptyset$ (note that at this step there is at most
one $\rho\in\Psi_1^{C^{(j)}\oplus G}(n)$). But $G\subseteq
Y_1\subseteq \pi_l^2(X')$ for some $X'\in [T_{V^{(1)}}]$, and by the definition of
$[T_{V^{(1)}}]$, $\forall n,[\Psi_1^{C^{(j)}\oplus G}(n)]^\preceq\cap
[V^{(1)}]^\preceq\ne\emptyset$ if $\Psi_1^{C^{(j)}\oplus
  G}(n)\!\downarrow$.
\item $[T_1]$ is a $\Pi_1^{0,C^{(j)}}$ class (note that
  $\textit{Cross}$ can be applied to binary strings), whose index can
  be $C^{(j)}$-computed from indices for $\Psi$, $V^{(1)}$, $V^{(2)}$,
  and $V^{(3)}$.
\item $\bigcup_{i=1}^6 Y_i=\omega$ (see Fact \ref{fac5}; this is just the
  pigeonhole principle). This is why we try to choose \emph{three}
  mutually disjoint clopen sets at \emph{this} step, i.e.\ if we
  choose, say, two disjoint clopen sets, then the union of all parts of
  some path through $T_1$ might be finite, which leads to difficulties
  at the
  next step of the construction. The advantage of requiring $T$ to be an
  ordered-partition-tree will be clear in Lemma \ref{fac7}. In general,
  such three mutually disjoint clopen sets can't be found, so instead
  we \emph{try} to find a $2$-disperse sequence of clopen sets for some
  $2$-supporter at \emph{this} step, while at later steps $2$-disperse
  sequences will not be enough, and will have to be replaced by
  $k'$-disperse sequences.
\end{enumerate}
\noindent
\emph{\textbf{How the requirement R(1;j) is satisfied}}: If case i occurs,
 we require that if finally $G\subseteq A(\bar{A})\wedge set(\rho)\subset A(\bar{A})$
 then $G\supset \rho$. If case ii occurs, we require that
  for some path $X\in [T_1]$, $G$ will be contained in some $Y_i$ such that $Y_i=\pi_i^6(X)$.
  By the above notes clearly $R(1;j)$ is satisfied. Furthermore, Lemma \ref{fac8} will prove that if
  no such $2$-disperse sequence exists such
  that all induced classes are non-empty (i.e.\ case ii does not occur),
  then case i occurs,  else one can compute a strong $k$-enumeration of $Q$.


\subsection{Tree forcing}
\label{subsec2}
A Mathias forcing condition is a pair $(\rho,L)$ where $\rho\in
2^{<\omega}$ and $L\in 2^\omega$. We write $(\rho,L)\geq(\rho',L')$
and say that $(\rho',L')$ extends $(\rho,L)$
iff $\rho'\supset\rho\wedge L'\cup \set(\rho')\subseteq
L\cup \set(\rho)$. A set $G$ satisfies $(\rho,L)$ iff $G\supset \rho\wedge G\subseteq \set(\rho)\cup L$.

In this paper, a tree forcing condition is a triple
$(\vec{\rho},T,k)$, where $T\subseteq 2^{<\omega}$ is a tree and
$\vec{\rho}=(\langle\rho_{1l},\rho_{1r}\rangle,\langle\rho_{2l},\rho_{2r}\rangle,\ldots,
\langle\rho_{kl},\rho_{kr}\rangle)$ is an array of $k$ pairs of binary strings, such that:
\begin{enumerate}
\item each $X\in [T]$ together with $\vec{\rho}$ codes in a uniform
  way $2k$ Mathias forcing conditions,
  $\langle\rho_{1l},\pi_1^k(X)\rangle,\langle\rho_{1r},\pi_1^k(X)\rangle,\ldots,
  \langle\rho_{kl},\pi_k^k(X)\rangle,\langle\rho_{kr},\pi_k^k(X)\rangle$,
  i.e., there is \emph{one} initial segment requirement pair for \emph{each} part.

\item $\bigcup_{i=1}^k \pi_i^k(X)=^*\omega$ (i.e.\ $\bigcup_{i=1}^k \pi_i^k(X)$
is $\omega$ minus some finite set).
\item $(\forall i\leq k,X\in [T])\ \pi_i^k(X)\cap\{1,2,\ldots,\max\{|\rho_{il}|,|\rho_{ir}|\}\}=\emptyset$.
\item $\set(\rho_{il})\subseteq A\wedge \set(\rho_{ir})\subseteq\bar{A}$
where $A\in2^\omega$ is the given set in Theorem \ref{th1}.
\end{enumerate}
\medskip

A set $G$ satisfies $(\vec{\rho},T,k)$ iff there exist $X\in [T]$ and
$i\leq k$ such that $G$ satisfies either $\langle\rho_{il},\pi_i^k(X)\rangle$ or
$\langle\rho_{ir},\pi_i^k(X)\rangle$.
$(\vec{\rho}\ ',T',k')$ extends $(\vec{\rho},T,k)$, written as
$(\vec{\rho}\ ',T',k')\leq(\vec{\rho},T,k)$, iff there exists a
function $p:\{1,2,\ldots, k'\}\rightarrow \{1,2,\ldots, k\}$ such that
\begin{enumerate}
\item $\forall i\leq k'\ \rho_{il}'\supset \rho_{p(i),l}\wedge \rho_{ir}'\supset\rho_{p(i),r}$
 \item $\forall i\leq k'\ \forall X'\in [T']\exists X\in [T]\,
   \set(\rho_{il}')\cup\pi_i^{k'}(X')\subseteq
   \set(\rho_{p(i),l})\cup\pi_{p(i)}^k(X)$ and similarly for
   the right-hand side. We call part $i$ of $T'$ a child part of part $p(i)$ of $T$.
\end{enumerate}
\medskip

We will construct a sequence
$(\vec{\rho}^{\ 1},T_{1},k_1)\geq(\vec{\rho}^{\ 2},T_{2},k_2)
\geq\cdots \geq(\vec{\rho}^{\ s},T_{s},k_s)\geq\cdots$ of tree forcing
conditions. Each part $i\leq k_s$ of $(\vec{\rho}^{\ s},T_s,k_s)$ will
correspond to a progress information sequence
$(j;e^{s,j}_{i,r},e^{s,j}_{i,l})$ for $j\in\mathbb{N}$, which means
that for all $i\leq k_s$ and $X\in[T_s]$, if $G$ satisfies
$\langle\rho^s_{il},\pi_i^{k_s}(X)\rangle$ then $G$ satisfies $R(e;j)$
for all $j\in\mathbb{N}$ and $e< e^{s,j}_{i,l}$, and if $G$ satisfies
$\langle\rho^s_{ir},\pi_i^{k_s}(X)\rangle$ then $G$ satisfies $R(e;j)$
for all $j\in\mathbb{N}$  and $e< e^{s,j}_{i,r}$. Furthermore, we will
have $[T_s]\ne\emptyset$, and each $T_s$ will be computable in some
$C\in\mathcal{M}$ where $\mathcal{M}$ is the given countable class in
Theorem \ref{th1}.

Let $C^{(j)}$ for $j \in \mathbb N$ be a sequence of sets of
$\mathcal{M}$ that is cofinal
in $\mathcal{M}$, i.e.\ $(\forall
j\in\mathbb{N})[C^{(j)}\in\mathcal{M}\wedge (\forall C\in
\mathcal{M}\exists C^{(j)},\ C\leq_T C^{(j)})]$, and such that
$\forall j\in\mathbb{N},\ C^{(j)}\leq_T C^{(j+1)}$. Fix a sequence
$C^{t_s}$ for $s \in \mathbb N$ of
elements of $\{C^{(j)}\}_{j\in\mathbb{N}}$ such that each $C^{(j)}$
appears infinitely often in this sequence,
i.e.\ $\exists^{\infty}s,t_s=j$. Step $s$ will be devoted to
$R(\cdots;t_s)$, which means we make sure that if part $i'$ of $T_s$ has a
child in $T_{s+1}$ then on each child of part $i'$, for example part
$i$ of $T_{s+1}$, for $j=t_s$ either $e^{s,j}_{ir}>e^{s-1,j}_{i'r}$ or
$e^{s,j}_{il}>e^{s-1,j}_{i'l}$.

To avoid too many indices we use formal parameters
$T,\vec{\rho},\vec{\Psi}$. Their values are updated at each step, so
$\vec{\Psi}=\vec{e}^{\ s-1,j}$ means we assign $e^{s-1,j}_{il}$ to be
the index of $\Psi_{il}$, and writing $T=\{\rho\in
2^{<\omega}:\ldots\sigma\in T\ldots\}$ makes sense. Within each step,
and from one step to the next, progress information may also need to
be updated, so writing $e^{s-1,j}_{il}=e^{s-1,j}_{il}+1$ or
$e^{s,j}_{il}=e^{s-1,j}_{il}+1$ also makes sense.

\subsection{Some definitions and facts}
\label{subsec3}
Before we get a glimpse at step $s$, we make the following definition
in order to simplify our statements:

\begin{definition}
\label{def3}
\noindent
\begin{enumerate}
\item $\Psi^{C\oplus\rho}$ \emph{abandons} $V$ on a set $Y$ iff there
exist $Z\subseteq Y$ and $n\in\mathbb{N}$ such that either $\Psi^{C\oplus
Z/\rho}(n)\!\downarrow=D$ for a finite set $D$ with $[D]^\preceq\cap
[V]^\preceq =\emptyset$, or $|\Psi^{C\oplus Z/\rho}(n)|$ is greater
than the index of $\Psi$.
\medskip

\item $\langle\Psi_l^{C\oplus\rho_l},\Psi_r^{C\oplus\rho_r}\rangle$ (also denoted by $\Psi^{C\oplus\rho}$) abandons $V$ on $X_1\oplus X_2$ iff \emph{either} $\Psi_l^{C\oplus\rho_l}$ abandons $V$ on $X_1$ \emph{or} $\Psi_r^{C\oplus\rho_r}$ abandons $V$ on $X_2$.
\medskip

\item $\Psi^{C\oplus\rho}$ (which means
$\langle\Psi_l^{C\oplus\rho_l},\Psi_r^{C\oplus\rho_r}\rangle$)
\emph{abandons} $V$ on $X$ iff for \emph{all} ordered 2-partitions
$X_1\cup X_2=X$,
$\langle\Psi_l^{C\oplus\rho_l},\Psi_r^{C\oplus\rho_r}\rangle$ abandons
$V$ on $X_1\oplus X_2$. I.e.,
$\langle\Psi_l^{C\oplus\rho_l},\Psi_r^{C\oplus\rho_r}\rangle$ doesn't
abandon $V$ on $X$ iff there exists an ordered partition $X_1\cup
X_2=X$ such that
$\forall Y\supset \rho_l,Y\subseteq X_1\cup
\set(\rho_l) \Rightarrow  (\forall n)\ [(\Psi_l^{C\oplus
  Y}(n)\!\uparrow)\vee ([\Psi_l^{C\oplus Y}(n)]^\preceq\cap
       [V]^\preceq\neq\emptyset\wedge |\Psi_l^{C\oplus Y}(n)|\leq l)]$ and
$\forall Y\supset \rho_r,Y\subseteq X_2\cup
\set(\rho_r) \Rightarrow (\forall n) [(\Psi_r^{C\oplus
  Y}(n)\!\uparrow)\vee ([\Psi_r^{C\oplus Y}(n)]^\preceq\cap
       [V]^\preceq\ne\emptyset \wedge |\Psi_r^{C\oplus Y}(n)|\leq r)]$.
\medskip

\noindent
(In the above definitions, ``abandons on a set'' can be replaced by
``abandons on a finite string''; the definitions generalize
naturally.)
\medskip

\item $\vec{\Psi}^{C\oplus\vec{\rho}}$ abandons $V$ on
  $X=\bigoplus_{i=1}^k X_i$ iff $\exists i\leq k$ such that the
  pair $\Psi_i^{C\oplus\rho_i}$ abandons $V$ on $X_i$. I.e.,
  $\vec{\Psi}^{C\oplus\vec{\rho}}$ doesn't abandon $V$ on
  $X=\bigoplus_{i=1}^k X_i$ iff for all $i$, there exist
  $X_{il},X_{ir}$ with $X_{il}\cup X_{ir}=X_i$ such that $\Psi_{il}^{C\oplus\rho_{il}}$ doesn't abandon $V$ on $X_{il}$ and $\Psi_{ir}^{C\oplus \rho_{ir}}$ doesn't abandon $V$ on $X_{ir}$.

    We also say that $\vec{\Psi}^{C\oplus\vec{\rho}}$ doesn't abandon
    $V$ on $\bigoplus_{i=1}^k (X_{il}\oplus X_{ir})$ to indicate that
    for all $i$, $\Psi_{il}^{C\oplus\rho_{il}}$ doesn't abandon $V$ on $X_{il}$ and $\Psi_{ir}^{C\oplus \rho_{ir}}$ doesn't abandon $V$ on $X_{ir}$. (Note that this is a little abuse of notation but the lower index $il$,$ir$ shall avoid confusion.)
Let $\pi_{il}^{k}(X)=X_{il}$ and $\pi_{ir}^{k}(X)=X_{ir}$.
\end{enumerate}
\end{definition}
The following two simple facts illustrate the central idea of the construction.
\begin{fact}
\label{fac2}
Let $V^{(1)},V^{(2)},\ldots, V^{(m)}$ be an $e$-disperse sequence of
clopen sets and let $\Psi$ be a Turing functional such that for each
$i$, $\Psi$ doesn't abandon $V^{(i)}$ on
$X$. Then, for any $Y\subseteq X$, $\Psi^Y$ is not total or is
not a strong $e$-enumeration. (The proof also holds if $\Psi$ is relativized,
i.e. it also holds if for some oracle $C$ and some $\rho$
we replace $\Psi$ by $\Psi^{C\oplus\rho}$ and
$\Psi^Y$ by $\Psi^{C\oplus Y/\rho}$. )
\end{fact}
\begin{proof}
Here and below $\Psi$ abandoning $V$ means $\Psi^{\emptyset\oplus\varepsilon}$
abandoning $V$, where $\varepsilon$ is the empty string.
Suppose not. Then there is some $n$ and some $Y\subseteq X$ such that $|\Psi^Y(n)|\leq
e$ and $$\forall
\rho\in\Psi^Y(n),\ |\rho|>\max\{\height(V^{(1)}),\height(V^{(2)}),\ldots,
\height(V^{(m)})\},$$ and furthermore for each $i\leq m$, $[V^{(i)}]^\preceq\cap
[\Psi^Y(n)]^\preceq\ne\emptyset$. For $1\leq i\leq e$, let
$P^{(i)}=\{V^{(j)}: [V^{(j)}]^\preceq\cap[\rho_i]^\preceq\ne\emptyset \textrm{,
  where } \rho_i \textrm{ is the } i^{\text{th}} \textrm{ string in }
\Psi^Y(n)\}$ (let $P^{(i)}=P^{(1)}$ if $i>|\Psi_e^Y(n)|$). Note that since
$|\rho_i| > \height{V^{(r)}}$, we have that $[\rho_i]^\preceq\cap [V^{(r)}]^{\preceq}\ne\emptyset$
implies $[\rho_i]^\preceq\subseteq [V^{(r)}]^{\preceq}$. The
$P^{(i)}$ form an $e$-partition of $V^{(1)},V^{(2)},\ldots, V^{(m)}$ but
$\forall i,\bigcap_{r\in
  P^{(i)}}[V^{(r)}]^\preceq\supseteq[\rho_i]^\preceq\ne\emptyset$,
which contradicts the assumption that
$V^{(1)},V^{(2)},\ldots,V^{(m)}$ is an $e$-disperse sequence.
\end{proof}
\begin{fact}
\label{fac3}
If a single Turing functional $\Psi_e^{C\oplus\rho}$ doesn't abandon $V$ on $X$ then for any $Y\subseteq X$, $\Psi_e^{C\oplus\rho}$ doesn't abandon $V$ on $Y$.
\end{fact}
\begin{proof}
By the definition of ``abandon'', item (1).
\end{proof}
Combining Fact \ref{fac2} and Fact \ref{fac3}:
\begin{fact}
\label{fac4}
Let $V^{(1)},V^{(2)},\ldots, V^{(m)}$ be an $e$-disperse sequence of
clopen sets and let $\Psi$ be a Turing functional such that for each
$i$, $\Psi$ as a strong $e$-enumeration doesn't abandon $V^{(i)}$ on
a set $X^{(i)}$. Then, for any $Y\subseteq \bigcap_{i=1}^{m}X^{(m)}$,
$\Psi^Y$ is not total or is not a strong $e$-enumeration. (As Fact \ref{fac2}
this also holds for the relativized version.)
\end{fact}

Together with Fact \ref{fac4}, the following two facts tell us how to
apply the $\textit{Cross}$ operation in order to ensure that $T_s$ is an
ordered-partition-tree.

\begin{fact}
\label{fac1}
Let $e_1,e_2,\ldots, e_u$ be $u$ many positive numbers and let
$k'=\sum_{i=1}^u e_i$. If $V^{(1)},V^{(2)},\ldots, V^{(n)}$ is a
$k'$-disperse sequence of clopen sets, for each $1\leq i\leq u$,
let $\mathcal{K}_i=\{K\subseteq \{1,2,\ldots, n\}: \{V^{(j)}\}_{j\in
K}  \textrm{ is an $e_i$-disperse class}\}$. Then
$\mathcal{\vec{K}}=(\mathcal{K}_1,\ldots, \mathcal{K}_u)$ is a
$u$-supporter of $\{1,2,\ldots, n\}$.
\end{fact}
\begin{proof}
Suppose, on the contrary, there is an ordered partition
$P^{(1)},P^{(2)},\ldots, P^{(u)}$ of $\{1,2,\ldots,n\}$ such that
for
all $1\leq i\leq u$, $P^{(i)}\notin \mathcal{K}_i$,
i.e.\ $\{V^{(j)}\}_{j\in P^{(i)}}$ is not an $e_i$-disperse sequence of
clopen sets. Then for each $i$, there exists a partition
$P^{(i,1)}\cup P^{(i,2)} \cup\cdots \cup P^{(i,e_i)}=P^{(i)}$ such
that $(\forall e\leq e_i)(\bigcap_{j\in P^{(i,e)}}[V^{(j)}]^\preceq
\ne\emptyset)$. However, then $P^{(1,1)},P^{(1,2)},\ldots,
P^{(1,e_1)},\ldots, P^{(i,j)},\ldots, P^{(K,e_u)}$ is a $k'$-partition
of $\{1,2,\ldots, n\}$ that contradicts the assumption that
$V^{(1)},V^{(2)},\ldots, V^{(n)}$ is a $k'$-disperse class of clopen
sets.
\end{proof}
By the definition of $u$-supporter we have:
\begin{fact}
\label{fac5}
Let $\mathcal{\vec{K}}=(\mathcal{K}_1,\ldots,\mathcal{K}_u)$ be a
$u$-supporter of $\{1,2,\ldots, n\}$, let the sets $X^{(1)}, X^{(2)}, \ldots,
X^{(n)}$ be ordered $u$-partitions of $W$, and let
$K'=\sum_{i=1}^{K}|\mathcal{K}_i|$. Then
$\textit{Cross}(X^{(1)},X^{(2)},\ldots, X^{(n)};\mathcal{\vec{K}})$ is
an ordered $K'$-partition of $W$.
\end{fact}
\begin{proof}
This is straightforward by the definition of $u$-supporter. We show
that for each $x\in W$ there exists an $i$ and $K_{i,j}\in\mathcal{K}_i$, such
that $x\in \bigcap_{p\in K_{i,j}}\pi_i^u(X^{(p)})$. Note that for each
$m\leq n$, $x$ belongs to some part of $X^{(m)}$, and therefore
$P^{(i)}=\{m\leq n: x\in\pi_i^u(X^{(m)})\}$, $1\leq i\leq u$, is a
$u$-partition of $\{1,2,\ldots, n\}$. By the definition of
$u$-supporter, there exists an $i$ and $K_{i,j}\in\mathcal{K}_i$ such that
$K_{i,j}\subseteq P^{(i)}$, which implies that $x\in\bigcap_{p\in
K_{i,j}}\pi_i^K(X^{(p)})$.
\end{proof}

Furthermore, the following plain fact tells us what kind of finite
extension of an initial segment requirement is allowed.
\begin{fact}
\label{fac9}
If $\Psi^{C\oplus \rho}$ doesn't abandon $V$ on $Y$ and
$\sigma\supset\rho$ is such that $\set(\sigma)\subseteq Y\cup
\set(\rho)$, then $\Psi^{C\oplus \sigma}$ also doesn't abandon
$V$ on $Y$. Furthermore, if $\vec{\Psi}^{C\oplus\vec{\rho}}$ doesn't
abandon $V$ on $X=\bigoplus_{i=1}^k X_k$ and
$\vec{\sigma}\supset\vec{\rho}$ is such that  $(\forall i\leq
k)\set(\sigma_i)\subseteq X_i\cup \set(\rho_i)$, then
$\vec{\Psi}^{C\oplus \vec{\sigma}}$ also doesn't abandon $V$ on $X$.
\end{fact}

\subsection{\textbf{STEP s}: construct $(\vec{\rho}^{\, s},T_s)$}
\label{subsec4}
Suppose we have a $C$-computable ordered $k_{s-1}$-partition tree
$T_{s-1}$, for some $C\in \mathcal{M}$ (here we mean a $k$-partition
of some
$W=^*\omega$; note that in above argument $T_1$ is a 6-partition tree
of $\omega$), and we need to construct $T_{s}$ ensuring that on
each child part, either the left-hand side or the right-hand side
steps forward, i.e., if $G$ satisfies
$(\rho_{il},\pi_i^{k}(X))$ for some $X\in [T_{s-1}]$ then
$\Psi_{il}^{C^{(j)}\oplus G}$ is not a strong
$e^{s-1,j}_{i,l}$-enumeration of $Q$ (where
$e^{s-1,j}_{i,l}$ is the index of $\Psi_{il}$ as updated at step
$s-1$), and if $G$ satisfies $(\rho_{ir}, \pi_i^{k}(X))$ for some
$X\in [T_{s-1}]$ then $\Psi_{ir}^{C^{(j)}\oplus G}$ is not a
 strong $e^{s-1,j}_{i,r}$-enumeration of $Q$.

\textbf{Case i}: If there exists $\rho$ and $i$ satisfying either
\begin{multline*}\rho\supset \rho_{il} \wedge (\exists n,\ |\Psi_{il}^{C^{(j)}\oplus \rho}(n)|>e^{s-1,j}_{il}\vee[\Psi_{il}^{C^{(j)}\oplus\rho}(n)\!\downarrow]\cap [Q]=\emptyset)\wedge\\ (\set(\rho)\subseteq A)\wedge
    (\exists X\in [T],\ \set(\rho)\subseteq\pi_i^k(X)/\rho_{il})\end{multline*}
or
\begin{multline*}\rho\supset \rho_{ir} \wedge (\exists n,\ |\Psi_{ir}^{C^{(j)}\oplus \rho}(n)|>e^{s-1,j}_{ir}\vee[\Psi_{ir}^{C^{(j)}\oplus\rho}(n)\!\downarrow]\cap [Q]=\emptyset)\wedge\\ (\set(\rho)\subseteq \bar{A})\wedge
    (\exists X\in [T],\
  \set(\rho)\subseteq\pi_i^k(X)/\rho_{ir}).
\end{multline*}
Then we can extend $\rho_{il}$ or $\rho_{ir}$ to $\rho$, which ensures
that either $\Psi_{il}^{C^{(j)}\oplus G}$ or $\Psi_{ir}^{C^{(j)}\oplus G}$ is
not a strong constant-bound-enumeration of $Q$, assuming $G$
satisfies $\langle\rho_{il},\pi_{i}^k(X)\rangle$ or
$\langle\rho_{ir},\pi_{i}^k(X)\rangle$, respectively, for some $X\in
[T]$. Note that this operation doesn't increase the total number of
parts, and that the new $\Pi_1^0$ class $[T]$ is still nonempty since
$\exists X\in [T]$ such that
$\set(\rho)\subseteq\pi_i^k(X)/\rho_{il}$ (or similarly for
the right-hand side).  During
the construction the above process will be repeatedly carried out within
step $s$ until case i fails to occurs, which must happen as proved in
Lemma \ref{fac6}.

\textbf{Case ii}: If case i fails to occur.
 Let $k'=\sum_{i=1}^k (e^{s-1,j}_{i,l}+e^{s-1,j}_{i,r})$, where $k =
 k_{s-1}$, and each $e^{s-1,j}_{i,r}$ is as
updated in previous applications of case i this stage, if any,
 i.e.\ the present progress for $R(\cdots;j)$, part $i$. \emph{Try} to find a
 $k'$-disperse sequence of clopen sets $V^{(1)},V^{(2)}\ldots,
 V^{(n)}$ (in step 1, $k'=1+1=2$), such that the following classes are
 nonempty for $m \leq n$:

\begin{multline*}
[T_{V^{(m)}}]=
\{X=\bigoplus_{i=1}^k (X_{il}\oplus X_{ir}): \bigoplus_{i=1}^k (X_{il}\cup X_{ir})\in [T_s]
\wedge \\ \vec{\Psi}^{C^{(j)}\oplus \vec{\rho}}
\text{ doesn't abandon }V^{(m)} \text{ on } \bigoplus_{i=1}^k
(X_{il}\oplus X_{ir})\}.
\end{multline*}
The above condition means that
\begin{multline*}
(\forall i\leq k \forall G\subseteq X_{il} \forall
  n)(|\Psi_{il}^{C^{(j)}\oplus G/\rho_{il}}(n)|\leq e^{s-1,j}_{i,l}
  \wedge \\ (\Psi_{il}^{C^{(j)}\oplus G/\rho_{il}}(n)\!\uparrow\vee
[\Psi_{il}^{C^{(j)}\oplus G/\rho_{il}}(n)]^\preceq\cap
[V^{(m)}]^\preceq\ne\emptyset)),
\end{multline*}
and similarly for the right-hand side. Note that each $[T_{V^{(m)}}]$
is a $\Pi_1^{0,C^{(j)}}$ class of ordered $2k$-partitions.

For all $1\leq i\leq 2k$, if $i$ is even then let $i'=\frac{i}{2}$ and
define $\mathcal{K}_i=\{K\subseteq \{1,2,\ldots, n\}:
\{V^{(m)}\}_{m\in K} \textrm{ is an $e^{s-1,j}_{i',l}$-disperse
  class}\}$; if $i$ is odd then let $i'=\frac{i+1}{2}$ and define
$\mathcal{K}_i=\{K\subseteq \{1,2,\ldots, n\}: \{V^{(m)}\}_{m\in K}
\textrm{ is an $e^{s-1,j}_{i',r}$-disperse class}\}$. Note that by
Fact \ref{fac1}, $\mathcal{\vec{K}}=(\mathcal{K}_1,\ldots,
\mathcal{K}_{2k})$ is a $2k$-supporter of $\{1,2,\ldots, n\}$.

Let
\[T_{s}=\textit{Cross}(T_{V^{(1)}},T_{V^{(2)}},\ldots, T_{V^{(n)}};\mathcal{\vec{K}}),\]
i.e.,
$$
\forall X\in [T_{s}],\ X=\bigoplus_{i\leq 2k}^{}(\bigoplus_{j\leq
  |\mathcal{K}_i|}Y_{K_{i,j}}),$$ where
$$Y_{K_{i,j}}=\bigcap_{p\in K_{i,j}}\pi_{i}^{2k}(X^{(p)}) \textrm{ for }
X^{(p)}\in [T_{V^{(p)}}].$$

Then we try to satisfy positive requirements by appropriately extending the present
initial segment requirements which were updated in previous applications of case i of
the P-Operation in the last stage.
Then we update progress information (i.e.\ update $\vec{e}^j$).

Note:
\begin{enumerate}
\item $[T_{V^{(m)}}]$ is a $\Pi_1^{0,C\oplus C^{(j)}}$ class.
\medskip

\item If $[V]^\preceq\supseteq [Q]$ then $[T_V]\ne\emptyset$, as
otherwise case i occurs.
\medskip

\item For all $i$, if $i$ is even then $\Psi_{il}^{C^{(j)}\oplus
  G/\rho_{il}}$ is not a strong
  $e^{s-1,j}_{il}$-enumeration of $Q$ for any
  $G\subseteq Y_{K_{i,j}}$, and if $i$ is odd then
  $\Psi_{ir}^{C^{(j)}\oplus G/\rho_{ir}}$ is not a strong
  $e^{s-1,j}_{ir}$-enumeration of $Q$ for any
  $G\subseteq Y_{K_{i,j}}$ (by Fact \ref{fac4} and definition of
  $\vec{\mathcal{K}}$), i.e.\ on each child part $R(\cdots;j)$ steps forward
  on either the left-hand side or the right-hand side.
\medskip

\item $[T_{s}]$ is a $\Pi_1^{0,C\oplus C^{(j)}}$ class;
\medskip

\item $\bigcup_{i\leq 2k}(\bigcup_{j\leq |\mathcal{K}_i|}
  Y_{K_{i,j}})=W$, i.e.\ $T_{s}$ is an ordered-$K'$-partition tree,
  where $K'=\sum_{i=1}^K |\mathcal{K}_i|$. (See Fact \ref{fac5}; this
  is why we choose a $k'$-disperse class of clopen sets at \emph{this} step.)
\end{enumerate}

\emph{\textbf{How the requirement R(e;j) is satisfied}}: For some path
$X\in [T_{s}]$, $G$ will be contained in some $\pi_{i'}^{K'}(X)$, and
we will have $G\supset \rho_{i'l}$ or $G\supset \rho_{i'r}$. Therefore
either $R(e^{s-1,j}_{il};j)$ or $R(e^{s-1,j}_{ir};j)$ is satisfied on
part $i'$ of $T_s$, where part $i$ of $T_{s-1}$ is the parent of part $i'$ of $T_s$.

\section{Construction and verification}
\label{sec4}
 Given a countable class $\mathcal{M}$ closed under disjoint union, a
 set $A$, and a closed set $[Q]$ of $2^\omega$ satisfying  the
 conditions stated in Theorem \ref{th1}, we give a concrete
 construction of an infinite set $G$ with $G\subseteq A\vee
 G\subseteq\bar{A}$, such that $\forall C\in\mathcal{M},G\oplus C$
 does not compute a constant-bound-enumeration of
  $Q$, and prove that the construction provides the
 desired $G$.

 If $A\leq_T C^{(j)}$ for some $C^{(j)}\in\mathcal{M}$ then just let
 $G=A$ if $A$ is infinite, and $G=\bar{A}$ otherwise. So without loss of generality assume $\forall C\in\mathcal{M},A\nleq_T C$.

\subsection{Some operations}
\label{subsec5}
Suppose we are given $A\in2^\omega$, and let $T$ be an ordered
$k$-partition tree of $2^{<\omega}$, with $\vec{\rho}$ being the
corresponding initial segment requirement.

To satisfy positive requirements we will apply the P-Operation.
\begin{definition}[\textbf{P-Operation}]
\label{}
P-Operation applied to the left-hand side of $T$'s part $i$: Choose
$\rho\supset\rho_{il}$ (if one exists) such that $[T_{\rho,i}]=\{X\in
[T]:set(\pi_i^k(X))\supset set(\rho)\}\neq\emptyset$ and $\emptyset\ne
\set(\rho)-\set(\rho_{il})\subset A$, then update initial segments: let
$\rho_{il}=\rho/\rho_{il}$, let $\rho_{jl}=\rho_{jl}$ for $j\neq i$,
and let $\rho_{ir}=\rho_{ir}$ for all $i$. Finally let $T = T_{\rho,i}$.

If such a $\rho$ exists, we say the P-Operation succeeds. Otherwise do
nothing; in this case we say the P-Operation fails.

The P-Operation applied to the right-hand side of $T$'s part $i$ is
analogous, with $A$ replaced by $\bar{A}$.
\end{definition}
To satisfy $R(\cdots;j)$ we will apply the R-i-Operation or the
R-ii-Operation depending on which of case i or case ii occurs.

\begin{definition}[\textbf{R-i-Operation} to
    $(T,\vec{\rho},j,\vec{\Psi}=\vec{e}^{\ s-1,j})$] Choose the least
  $i\leq k$ for which either

  \begin{itemize}
  \item there exists a $\rho\supset\rho_{il}$ such that
\begin{multline*}
(\exists n,|\Psi_{il}^{C^{(j)}\oplus
    \rho}(n)|>e^{s-1,j}_{il}\vee[\Psi_{il}^{C^{(j)}\oplus\rho}(n)\!\downarrow]\cap
  [Q]=\emptyset)\wedge \\ (\set(\rho)\subseteq A)\wedge
    (\exists X\in [T],
  \set(\rho)\subseteq\pi_i^k(X)/\rho_{il}),
\end{multline*}
or

\item there exists a $\rho\supset\rho_{ir}$ satisfying a similar
condition as above but with $\set(\rho)\subseteq A$ replaced by $\set(\rho)\subseteq \bar{A}$
and ``$il$'' replaced by ``$ir$''.
\end{itemize}
\medskip

If such a $\rho$ exists we say case i occurs to $R(\cdots;j)$. For each
subcase we say that the left-hand side of part $i$ steps forward and
that $R(\cdots;j)$ steps forward on the left-hand side; or that the
right-hand side of part $i$ steps forward and that $R(\cdots;j)$ steps
forward on the right-hand side. We say that $R(\cdots;j)$, part $i$ steps
forward no matter which side steps forward.
\medskip

Then update initial segment requirements (according to $\rho,i$), i.e.\ set
$\rho_{il}(\rho_{ir}) = \rho$ in the first (second) subcase,
and update progress information:
\begin{align*}
    e^{s-1,m}_{il}=e^{s-1,m}_{il}  \textrm{ and } e^{s-1,m}_{ir}=e^{s-1,m}_{ir}&\text{ if } m\ne j \\ e^{s-1,j}_{ir}=e^{s-1,j}_{ir}+1  \textrm{ and } e^{s-1,j}_{il}=e^{s-1,j}_{il} &\text{ if the right-hand side steps forward} \\
    e^{s-1,j}_{il}=e^{s-1,j}_{il}+1  \textrm{ and } e^{s-1,j}_{ir}=e^{s-1,j}_{ir} &\text{ if the left-hand side steps forward}
\end{align*}
and $\vec{\Psi}=\vec{e}^{\ s-1,j}$ in any case.

\end{definition}

\begin{definition}[\textbf{R-ii-Operation} to $(T,\vec{\rho},j,\vec{\Psi}=\vec{e}^{\ s-1,j})$]

For a clopen set $W$, let
 \begin{multline*}
[T_{W}]=
    \{X=\bigoplus_{i=1}^k (X_{il}\oplus X_{ir}): \bigoplus_{i=1}^k (X_{il}\cup X_{ir})\in [T]
    \wedge \\ \vec{\Psi}^{C^{(j)}\oplus \vec{\rho}}
    \text{ doesn't abandon }W \text{ on } \bigoplus_{i=1}^k (X_{il}\oplus X_{ir})\}\ne\emptyset.
    \end{multline*}

In this operation we find a $k'$-disperse sequence of
clopen sets $V^{(1)},V^{(2)},\ldots, V^{(n)}$ where
$k'=\sum_{i=1}^{k}e^{s-1,j}_{il}+e^{s-1,j}_{ir}$ such that for all $m \leq n$,
$[T_{V^{(m)}}]\ne \emptyset$.

If such a $k'$-disperse sequence can be found, we say case
    ii occurs. We will show in Lemma \ref{fac8} that if case i fails to
    occur then case ii must occur, as otherwise we obtain a
    $C^{(j)}$-computable constant-bound-enumeration of $Q$.
\medskip

Define $\mathcal{\vec{K}}$: for $i=1,2,\ldots, 2k$,
     \begin{enumerate}
     \item if $i$ is even then let $i'=\frac{i}{2}$ and define
       $$\mathcal{K}_i=\{K\subseteq \{1,2,\ldots, n\}:
       \{V^{(m)}\}_{m\in K} \textrm{ is an $e^{s-1,j}_{i',l}$-disperse class}\};$$
     \item if $i$ is odd then let $i'=\frac{i+1}{2}$ and define
       $$\mathcal{K}_i=\{K\subseteq \{1,2,\ldots, n\}:
       \{V^{(m)}\}_{m\in K}  \textrm{ is an $e^{s-1,j}_{i',r}$-disperse class}\};$$
         \end{enumerate}
and let $u=\sum_{i=1}^{2k}|\mathcal{K}_i|$.
\medskip

      Update $T$:  $T=T_s=\textit{Cross}(T_{V^{(1)}},T_{V^{(2)}},\ldots, T_{V^{(n)})};\mathcal{\vec{K}})$;
      Update $k=k_s=u$.

      \noindent(Note that each new part is contained in an old part of
      some path through $T$. Call a new part $i'$ a left child of an old
      part $i$ iff for any $X\in [T_{s}]$, $\pi_{i'}^{k_{s}}(X)$ is the
      intersection of $\pi_{il}^{k_{s-1}}(X^{(p)})$, $p\in K_{2i,m}$
      for some $K_{2i,m}$, where $X^{(p)}\in [T_{V^{(p)}}]$; and
      similarly for the right-hand side with $2i$ replaced by $2i-1$.)
\medskip

      Update progress information: for $i=1,2,\ldots, u$, let
      \begin{align*}
    e^{s,m}_{il}=e^{s-1,m}_{il},\ e^{s-1,m}_{ir}=e^{s-1,m}_{ir}&\text{ if } m\ne j \\ e^{s,j}_{ir}=e^{s-1,j}_{i'r}+1,\ e^{s-1,j}_{il}=e^{s-1,j}_{i'l} &\text{ if part $i$ of }T_{s} \text{ is a right child of part $i'$ of }T_{s-1} \\
    e^{s,j}_{il}=e^{s-1,j}_{i'l}+1,\ e^{s-1,j}_{ir}=e^{s-1,j}_{i'r} &\text{ if part $i$ of }T_{s} \text{ is a left child of part $i'$ of }T_{s-1}.
     \end{align*}
Finally, assign the value $\vec{e}^{\ s,t_{s+1}}$ to $\vec{\Psi}$.
\end{definition}

\subsection{Construction}
\label{subsec6}
Let $C^{(j)}$ be a sequence of sets of $\mathcal{M}$ that is cofinal
in $\mathcal{M}$, i.e.\ $(\forall
j\in\mathbb{N})[C^{(j)}\in\mathcal{M}\wedge (\forall C\in
\mathcal{M}\exists C^{(j)},C\leq_T C^{(j)})]$; furthermore we require
that $\forall j\in\mathbb{N},C^{(j)}\leq_T C^{(j+1)}$. Fix a sequence
of elements of $\{C^{(j)}\}_{j\in\mathbb{N}}$, namely $C^{(t_s)}$,
such that each $C^{(j)}$ appears infinitely often in this sequence,
i.e.\ $\exists^{\infty}s,t_s=j$.

Now suppose we are at the beginning of step $s$, we have a
$C$-computable $k_{s-1}$-partition tree $T=T_{s-1}$ with
$[T]\ne\emptyset$, and $C^{(j)}$ is the corresponding set in the
$s^{\text{th}}$ position of the sequence (i.e.\ $t_s=j$, so step $s$
is devoted to $R(\cdots;j)$), with current progress
$\vec{\Psi}=\vec{e}^{\ s-1,j}$ for $R(\cdots;j)$ and initial segment
requirements $\vec{\rho}$.

\emph{Begin step $s$ of the construction:}
\begin{enumerate}
 \item Repeatedly apply the R-i-Operation to
   $(T,\vec{\rho},j,\vec{\Psi}=\vec{e}^{\ s-1,j})$ until case i fails
   to occur to $R(\cdots;j)$.

 \noindent (We will show in Lemma \ref{fac6} that after finitely many
 repetitions case i will stop occurring).

 \item Then apply the R-ii-Operation to
   $(T,\vec{\rho},j,\vec{\Psi}=\vec{e}^{\ s-1,j})$.

 \noindent (We will show in Lemma \ref{fac8} that if case i doesn't
occur then case ii must occur, as otherwise we obtain a
 strong constant-bound-enumeration of $Q$.)

\item For each part $i$, $i\leq k_s$, if part $i$ is a left child then
  apply the P-Operation to the left-hand side of $T_s$'s part $i$;
  otherwise apply the P-Operation to the right-hand side of $T_s$'s
  part $i$.

     \noindent (Here the P-Operation is applied to each part in order of
     their indices. We will show in Lemma \ref{fac7} that the
     P-Operation succeeds on at least one of the new parts. Note that
     on each ``new'' part at least one side of $R(\cdots;j)$ has just
     stepped forward.)

 \item Finally, go to the next step.
\end{enumerate}
\medskip

 Note that in either case $R(\cdots;j)$ only steps forward but not ``backward'', since each child part is a subset of its parent part, and the initial segment requirements are inherited.

\subsection{Verification}
\label{subsec7}
We need the following three facts as mentioned above.
\begin{lemma}
\label{fac6}
Within each step, case i will not occur to $R(\cdots;j)$ forever.
\end{lemma}

\begin{lemma}
\label{fac7}
For any $A$ and any $C$-computable ordered $k$-partition tree $T$ such
that $[T]\ne\emptyset$, if $A\nleq_T C$ then there exist $i\leq k$ and
$X \in [T]$ such that both $\pi_i^k(X)\cap A$ and $\pi_i^k(X)\cap
\bar{A}$ are nonempty, which implies that in case ii, for at least one
of the new parts, the P-Operation succeeds on the side of that part to
which it is applied.
\end{lemma}

\begin{lemma}
\label{fac8}
Let $E$ be a set of finite subsets of $2^{<\omega}$ c.e.\ in $\mathbf{e}$
such that, for all $n$:
\begin{enumerate}
\item $\mathcal{W}_n=\{W\subseteq 2^{<\omega}: \forall \rho\in
  W,|\rho|=n\wedge W\notin E\} $ is not a $k'$-disperse class.
\item $Q_n=\{\rho\in 2^{<\omega}: |\rho|=n\wedge [\rho]\cap
  [Q]\ne\emptyset\}\notin E$.
\end{enumerate}
Then there exists a strong $k'$-enumeration $h:\omega\rightarrow \omega$
of $Q$ computable in $\mathbf{e}$.

 Actually, we could further
require that $(\forall \rho\in D_{h(n)})|\rho|=n$.
This implies that for any specific tuple $(T,\vec{\rho},j,\vec{\Psi}=\vec{e}^{\ s-1,j})$
 in the construction either case i or case ii occurs since otherwise
$E=\{W\subseteq 2^{<\omega}: \forall \rho,\sigma\in W,|\rho|=|\sigma|,
[T_W]=\emptyset\}$ is a $C\oplus C^{(j)}$-c.e.\ set which satisfies
(1) and (2) (recall item (2) in the Note to Step $s$ and the
definition of R-ii-Operation).
\end{lemma}

Assuming these facts we now show that the construction does provide
the desired $G$. Note that our construction yields a tree
$\mathcal{T}$: Each node at the $s^{\text{th}}$ level represents a
part at the $s^{\text{th}}$ step, and corresponds to a tuple
$(j;e_l,e_r)$ standing for the progress of step $s$ for $C^{(j)}$ (on
this part), where $C^{(j)}$ is the set that this step is devoted to,
i.e., the $j$ such that $R(\cdots;j)$ steps forward at this step. The
successor nodes of a
node $c$ are those representing parts that are descendant parts
of the part $c$ represents. Clearly $\mathcal T$ is a finitely
branching tree.

\begin{definition}
Say part $i$ fades away on the left-hand side at step $s$ iff $\forall
X\in[T_s],\ \pi_i^k(X)\cap A=\emptyset$ (and similarly for the
right-hand side with $\bar{A}$ replacing $A$).
\end{definition}

Note that if part $i$'s left-hand side fades away at step $s$ then it
fades away  forever, i.e., all its descendants also
fade away at the left-hand side. So by Lemma \ref{fac7}, there
exists an infinite subtree $\mathcal{T'}$ of $\mathcal{T}$
such that each $s^{\text{th}}$ level node of $\mathcal{T'}$
represents a part that has not yet faded away on either side at step
$s$. Therefore there exists a path $f$ through this subtree. Note that
for each $R(\cdots;j)$ there exists some side, say the right-hand side, such
that along $f$, $R(\cdots;j)$ steps forward on this side infinitely
often. Therefore
there exists some side, say the right-hand side, such that for any
$C^{(j)}$ there exists $C^{(j')}$ with $C^{(j)}\leq_T C^{(j')}$, and
there exist infinitely many steps $s_m,m\in\mathbb{N}$, such that step
$s_m$ is devoted to $R(\cdots;j')$, $f(s_m)$'s represented part
$f_{s_m}$ steps forward at the right-hand side at step $s_m$, and the
following P-Operation (applied to the new part corresponding to this
part's right-hand side) succeeds. In other words, for any
$C\in\mathcal{M}$ there exists $C^{(j')}$ such that $C\leq_T
C^{(j')}$ and along $f$, $R(\cdots;j')$ steps forward on the right-hand side
infinitely often
(which implies that all $R(j;e_r)$, $e_r\in\mathbb{N}$ are satisfied) and
the P-Operation succeeds (on the right-hand side) infinitely
often. Therefore,
$G=\bigcup_{s=1}^\infty \rho^{s}_{f_s, r}$ is infinite and
$C^{(j')}\oplus G$ doesn't compute a
constant-bound-enumeration of $Q$,
 so that $G\oplus C\leq_T G\oplus
C^{(j')}$ also fails to compute a strong
constant-bound-enumeration of $Q$, i.e.\ $\forall
C\in\mathcal{M},G\oplus C$ doesn't compute a
strong constant-bound-enumeration of $Q$.

It remains to prove the three facts mentioned above.

\begin{proof}[Proof of Lemma \ref{fac6}]
Recall that there are infinitely many Turing functionals $\Psi_e$ such
that, for any oracle $C \oplus G$, $\Psi_e^{C\oplus G}$ doesn't halt
on any input. Call such Turing functionals trivial. Clearly, at step
$s$, if a part proceeds to $(e_l,e_r)$ (during the case i loop), where
$\Psi_{e_l}$ is a trivial Turing functional, then the R-i-Operation will
never succeed on the left-hand side of this part, and similarly for
the right-hand side if $\Psi_{e_r}$ is trivial. Since there are only
finitely many different parts and sides, and the
R-i-Operation doesn't increase the total number of parts, either case
i fails to occur during the loop or all Turing functionals in the
progress tuples will finally be updated to trivial ones,
which also causes case i to fail to occur,
 since case i occurring implies that at least some Turing functional halts on some
 input.
\end{proof}

\begin{proof}[Proof of Lemma \ref{fac7}]
Suppose on the contrary that $\forall i\leq k([\forall
X\in[T],\pi_i^k(X)\cap A=\emptyset]\vee [\forall X\in[T],\pi_i^k(X)\cap
\bar{A}=\emptyset])$. Let $R=\{i\leq k: \forall
X\in[T],\pi_i^k(X)\cap A=\emptyset\}$ and $L$ be the set of all $i
\leq k$ that are not in $R$. We show that $A$ is computable in $C$. Note that
$T$ is a partition of $W =^* \omega = \omega - F$
for some finite set $F$,
i.e.\ $\forall X\in[T],\bigcup_{i=1}^k \pi_i^k(X)\supseteq \omega-F$.
To decide whether $n \in A$, suppose $n\notin F$,
wait (using $C$) for a moment such that for some $m$,
$\forall \rho, (|\rho|=m\wedge \rho\in T)\Rightarrow[(\forall i\in L,
n\notin \pi_i^k(\rho))\vee (\forall i\in R, n\notin
\pi_i^k(\rho))]\vee[\exists N\in\mathbb{N},\forall \sigma\supset
\rho,|\sigma|=N\rightarrow \sigma\notin T]$. Note that such a moment
must exist by the assumption that each part is contained in either
$A$ or $\bar{A}$, so $n\in A\rightarrow [(\forall i\in R, n\notin
\pi_i^k(\rho))\vee(\exists N\in\mathbb{N},\forall \sigma\supset
\rho,|\sigma|=N\rightarrow \sigma\notin T)]$ and similarly for $n\in
\bar{A}$.
Furthermore, since $\forall X\in[T],\bigcup_{i=1}^k \pi_i^k(X)\supseteq \omega-F$
we have ($\forall i\in L, n\notin \pi_i^k(\rho))\rightarrow
n\in\pi_j^k(\rho)$ for some $j\in R$ and $\rho\subset X \in [T]$. Therefore $n\in
A$ iff
$\exists m\forall \rho, (|\rho|=m\wedge \rho\in T)\Rightarrow(\forall
i\in R, n\notin \pi_i^k(\rho))$ and $n\in \bar{A}$ iff $\exists m\forall
\rho, (|\rho|=m\wedge \rho\in T)\Rightarrow(\forall i\in L, n\notin
\pi_i^k(\rho))$. So we could compute $A-F$ using $C$, which is clearly
equivalent to $A\leq_T C$.
\end{proof}

\begin{proof}[Proof of Lemma \ref{fac8}]
We first show that if neither case i nor case ii occurs then
$E=\{W\subseteq 2^{<\omega}: \forall \rho,\sigma\in W,|\rho|=|\sigma|,
[T_W]=\emptyset\}$ is $C\oplus C^{(j)}$-c.e.\ and satisfies (1) and (2).
Note that $E$ is clearly $C\oplus C^{(j)}$-c.e.

Clearly, if case ii does not
occur then $E$ satisfies (1) by the definition of ``case ii
occurs''. Suppose for some $Q_n$, $[T_{Q_n}]=\emptyset$. Consider an
arbitrary $X\in [T]$. For $i\leq k$ let $X_{il}=\pi_i^k(X)\cap A$ and
$X_{ir}=\pi_i^k(X)\cap\bar{A}$. Since $[T_{Q_n}]=\emptyset$,
$X'=\bigoplus_{i\leq k} (X_{il}\oplus X_{ir})\notin
[T_{Q_n}]$. Therefore on some part $i$, $\vec{\Psi_{i}}^{C^{(j)\oplus
    \vec{\rho_i}}}$ abandons $Q_n$ on $X_{il}\oplus X_{ir}$. Thus, by
the definition of abandoning, item (2), either $(\exists
\rho\supset\rho_{il}, \set(\rho)\subseteq X_{il}\subseteq
A,\exists n\in\mathbb{N})[\Psi_i^{C^{(j)}\oplus \rho}(n)\!\downarrow]$
$\cap[Q_n]=\emptyset\vee|\Psi_{il}^{C^{(j)}\oplus\rho}(n)|> \textrm{index of
$\Psi_{il}$)}$ or $(\exists \rho\supset\rho_{ir},
\set(\rho)\subseteq X_{ir}\subseteq\bar{A},\exists
n\in\mathbb{N})[\Psi_{ir}^{C^{(j)}\oplus \rho}(n)\!\downarrow] \cap
[Q_n]=\emptyset\vee|\Psi_{il}^{C^{(j)}\oplus\rho}(n)|> \textrm{(index of
$\Psi_{il}$)}$. However, this implies case i occurs.

Now we show that the algorithm of $E$ (or equally speaking, the degree $\mathbf{e}$)
 can be used to compute a strong $k$-enumeration of $Q$.
Let $E_t$ denote $E$ at stage $t$ and
$\mathcal{W}_{n,t}=\{W\subseteq 2^{<\omega}:\forall \rho\in W, |\rho|=n\wedge W\notin
E_{t}\}$. Note that since $E$ is an enumeration,
$\mathcal{W}_{n,t+1}\subseteq \mathcal{W}_{n,t}$. Therefore, if there
exists a $k'$-partition of $\mathcal{W}_n$, $P^{(1)}\cup
P^{(2)}\cup\cdots\cup P^{(k')}=\mathcal{W}_n$, such that $(\forall
i)(\bigcap_{W\in P^{(i)}}W\ne\emptyset)$, then this partition can be found in a
finite amount of time. Furthermore, $Q_n\in P^{(i)}$ for some $i$, so
$\forall \rho\in\bigcap_{W\in P^{(i)}}W, \rho\in Q$. It follows that
the function $h$ defined below, which is computable using an algorithm for $E$ (i.e.\ computable in any degree that computes $E$),
is a strong $k'$-enumeration of $Q$:
\begin{multline*}
D_{h(n)}=\{ \rho_1,\rho_2,\ldots, \rho_{k'}: \text{there exists some } t \\
\text{and a }k'\text{-partition of }\mathcal{W}_{n,t}=\bigcup_{i=1}^{k'}P^{(i)} \text{ s.t.\ }
 (\forall j)(\bigcap_{W\in P^{(j)}}W\ne\emptyset);\\
   \text{for }i=1,2,\ldots, k',\ \rho_i \text{ is the leftmost string in }\bigcap_{W\in
    P^{(i)}}W\}.
\end{multline*}
\end{proof}

\section{Applications}
\label{sec5}
\begin{corollary}
\label{cor1-2}
Over $\mathsf{RCA}_0$, $\mathsf{RT_2^2}$ \text{ does not imply } $\mathsf{ WWKL_0}$
\end{corollary}

Before proving this corollary, we
make some remarks.
It is known that in $\omega$-models $\mathsf{WWKL_0}$ is equivalent to the assertion that
for each $X$ there is a set that is Martin-L\"{o}f random relative to $X$. It is known that
$X\in2^\omega$ is Martin-L\"{o}f random iff $(\forall
n)\ K(X\upharpoonright n)\geq^+ n$, where $K(\rho)$ denotes the prefix-free
Kolmogorov complexity of a string, and furthermore there exists a universal prefix-free machine $U$ such that for every prefix-free machine $M$, $K_U\leq K_M + c$ for some constant depending on $M$.
For more background on algorithmic randomness theory, see \cite{nies2009computability,downey2008algorithmic}.
Fix a universal prefix-free machine $U$ in the following text.
\begin{definition}
A string $\rho$ is $c$-incompressible$_U$ iff $K_U(\rho)\geq |\rho|-c$.
\end{definition}
Note that $[T_c]=\{X\in 2^{\omega}:\forall n,X\upharpoonright n \text{ is c-incompressible}_U\}$ is a closed set of $2^{\omega}$.

\begin{lemma}
\label{lem1}
For any positive integers $k,c$, there is no computable
strong $k$-e\-nu\-me\-ra\-tion of the $c$-incompressible$_U$ strings.
\end{lemma}
\begin{proof}
Suppose on the contrary $h:\omega\rightarrow\omega$ is a computable
strong $k$-enumeration of the $c$-incompressible$_U$
strings. Let
$\sigma_n=\underbrace{00\ldots 0}_{n}1$. We define a prefix-free
machine $M$. The algorithm of $M$ is as
follows for each $n$: look for an $m\geq n^2$, then
output $\rho_i$ on input $\sigma_{kn+i}$ (note that $j\leq k$). Note that all strings
$\rho_1,\ldots,\rho_j \in D_{h(m)}$ have length $m = n^2$. Since
$U$ is universal, there exists a constant $d$ such that
$K_{U}(\rho_i)<kn+i+d$, therefore for sufficiently large $n$, $\forall
i\leq k, K_{U}(\rho_i)<|\rho_i|-c$, contradicting the assumption that
$h$ is a strong $k$-enumeration of the
$c$-incompressible$_U$ strings.
\end{proof}
Note that although Theorem \ref{th1} proves the cone avoidance result for
\emph{one} closed set $[Q]$, it can be easily adapted to construct $G$
cone avoiding countably many closed sets as long as for every finite number
of them, $Q_1, Q_2,\ldots, Q_n$, the joint union of them,
$[\bigvee_{i = 1}^{n} Q_i] = \{X\in 2^\omega: (\exists j\leq n\exists Y\in [Q_j])X = \rho_j * Y \}$
where $\rho_j$ is determined by $Q_j$,
satisfies the condition in Theorem \ref{th1}.

Now we can prove the main theorem.
\begin{proof}[Proof of Corollary \ref{cor1-2}]
It suffices to construct a countable class $\mathcal{M}$ satisfying
four conditions: (a) $C,B\in\mathcal{M}\rightarrow C\oplus
B\in\mathcal{M}$; (b) $(C\in\mathcal{M}\wedge B\leq_T C)\rightarrow
B\in\mathcal{M}$; (c) $(\forall C\in\mathcal{M})$ $C$ does not compute
a strong constant-bound-enumeration of any $T_c$;
(d) $\mathcal{M}\vDash \mathsf{RT_2^2}$.

Note that (a) and (b) ensure that $\mathsf{RCA}_0$ is satisfied,
(c) ensures that $\mathsf{WWKL}_0$ is not satisfied, and (d) ensures that $\mathsf{RT}_2^2$ is satisfied.

It is shown in \cite{Cholak2001} Lemma 7.11 and later corrected in \cite{CholakErrata}
that $\mathsf{RCA_0+RT_2^2}$ is equivalent to
$\mathsf{RCA_0+SRT_2^2+COH}$.
Here $\mathsf{COH}$ is the axiom saying that
for any $C = \bigoplus_{i = 0}^\infty C_{i}\in \mathcal{M}$
there exists an infinite set $C^*$ cohesive for $C$, i.e.\ for any component
$C_i$ either $C^* - C_i$ or $ C^* - \bar{C_i}$ is finite.

Moreover, we now show that it can be proved by the finite extension method that
if $\{R^{(i)}\}_{i\in\mathbb{N}}$ is a sequence such that no finite disjoint
union of the $R^{(i)}$ computes a strong
constant-bound-enumeration of any $T_c$, then
there exists an infinite set $G$ cohesive for
$\{R^{(i)}\}_{i\in\mathbb{N}}$ that also does not compute any
strong constant-bound-enumeration of any $T_c$.
To construct such a cohesive set, we try to satisfy the following requirements.

\[
\begin{split}
& R_e: \Psi_e^G \text{ is not a strong $e$-enumeration of any } T_c \\
& N_i: G\subseteq^* R^{(i)}\vee G\subseteq^* \omega-R^{(i)}
\end{split}
\]

Clearly we can inductively define a 0-1 sequence $\nu$ such that, letting $R^i = R^{(i)}$ if $\nu(i) = 1$ and
$R^i = \omega - R^{(i)}$ if $\nu(i) = 0$, we have $(\forall n) |\bigcap_{i = 0}^n R^i| = \infty$.

For each step $s$, let $G_s = \bigcap_{i = 0}^n R^i$; we construct an initial segment $\rho_s$
such that $\rho_s\supset \rho_{s-1}$, $set(\rho_s)\subseteq set(G_{s-1}/\rho_{s-1})$,
 and furthermore, $(\rho_s, G_s)$, as a Mathias forcing condition, forces $R_s$ and $N_s$.
 It is clear that for any $\rho$, $(\rho, G_s)$ forces $N_s$. We now choose $\rho\supset\rho_{s-1}\wedge set(\rho)\subseteq G_{s-1}$
 to satisfy $R_s$. There are two ways to do so: find a  $\rho$ with $\rho\supset\rho_{s-1}\wedge set(\rho)\subseteq G_{s-1}$ such
 that $\Psi_e^\rho(n)$ does not halt, or one such that $|\Psi_e^\rho(n)\!\downarrow|> e\vee \Psi_e^\rho(n)\cap T_c = \emptyset\vee ((\exists \sigma\in\Psi_e^\rho(n))|\sigma|>n)$.
 Suppose we can't find such $\rho$; this  implies that for every $n$,
 there exists $\rho$ with $\rho\supset\rho_{s-1}\wedge set(\rho)\subseteq G_{s-1}$
 such that $|\Psi_e^\rho(n)\!\downarrow|\leq e$ and $\Psi_e^\rho(n)$ is
 a set of length $n$ strings that has nonempty intersection with $T_c$.
 But then one could use this $\Psi_e$ and oracle $\bigoplus_{i=0}^s R^{(i)}$
 to compute a strong $e$-enumeration of $T_c$, a contradiction.

In the following text, for a class of sets $\mathcal{M}$, we say $X$ is
$\mathcal{M}$-computable iff $\exists C\in\mathcal{M}$ s.t.\ $X\leq_T C$.
The notions of  $\mathcal{M}$-effectively closed, $\Pi^{0,\mathcal M}_1$, etc.\ are defined similarly.
Note that the conclusion just proved above is relativizable to
any countable class of sets closed under join. Let $\mathcal{M}_0=$ all
recursive sets. Let $f$ be a computable
stable coloring. By
Theorem \ref{th1}, there exists an infinite $G_0$ with $G_0\subseteq
f_1\vee G_0\subseteq f_2$ such that $\forall
C\in\mathcal{M}_0$, for each $c$, $C\oplus G_0$ does not compute any
strong constant-bound-enumeration of any infinite subset of any
 $T_{c}$. Let
$\mathcal{M}_1=\{X\in2^\omega: X\leq_T C\oplus G_0$ for some
$C\in\mathcal{M}_0\}$. Clearly $\mathcal{M}_1$ satisfies (a)(b)(c). Let
$G_1$ be cohesive for a sequence of uniformly
$\mathcal{M}_1$-computable sets (where $\mathcal{M}_1$-computable
means computable in some $C\in\mathcal{M}_1$), so that $(\forall
C\in\mathcal{M}_1)$ $G_1\oplus C$ does not compute any
strong constant-bound-enumeration of any infinite subset of any $T_c$. Let
$\mathcal{M}_2=\{X\in 2^\omega: \exists C\in\mathcal{M}_1,X\leq_T
C\oplus G_1\}$. Clearly $\mathcal{M}_2$ also satisfies
(a)(b)(c). Iterate the above process to ensure that (1) for any
uniformly $\mathcal{M}_j$-computable sequence $C^{(1)},C^{(2)},\ldots$,
there exists an $i$ and a $G_{i-1}\in\mathcal{M}_i$
such that $G_{i-1}$ is cohesive for
$C^{(1)},C^{(2)},\ldots$ and (2) for any
$C\in\Delta_2^{0,\mathcal{M}_j}$ ($\Delta_2^{0,\mathcal{M}_j}$ means
$\Delta_2^{0,C}$ for some $C\in \mathcal{M}_j$),
there exists an $i$ and an infinite
$G_{i-1}\in\mathcal{M}_i$ such that
$G_{i-1}\subseteq C\vee G_{i-1}\subseteq
\bar{C}$. It follows that
$\mathcal{M}=\bigcup_{i=0}^{\infty}\mathcal{M}_i
\vDash\mathsf{RCA_0+SRT_2^2+COH}\Leftrightarrow
\mathsf{RCA_0+RT_2^2}$ but clearly $\mathcal{M}$ satisfies
(a)(b)(c). The conclusion follows.
\end{proof}

We now prove Corollary \ref{cor4}.

\begin{corollary}
There exists a $\textit{DNR}$ function that does not compute
\emph{any} binary sequence with positive effective Hausdorff dimension.
\end{corollary}
\begin{proof}
Again note the construction can be adapted to construct a $G$ cone
avoiding countably many closed sets every finite joint union
of which satisfies the conditions of
Theorem \ref{th1}. Therefore, choose a sequence $d_n>0$ with
$\lim_{n\rightarrow \infty}d_n=0$ and let $A$ be a $\Delta_2^0$
set such that every infinite subset of $A$ or $\bar{A}$ computes a
$\textit{DNR}$ function. Now apply Theorem \ref{th1} to construct
$G\subseteq A\vee G\subseteq \bar{A}$ that cone avoids all
corresponding closed sets $[T_{d_n}]$, and let $f\leq_T G$ be a
$\textit{DNR}$ function. Then $f$ does not compute any binary sequence of positive effective Hausdorff dimension.
\end{proof}

For most ``natural'' trees that have been studied, if the induced
closed set has Muchnik degree strictly above $0$ as a mass problem,
then the tree does not admit a
strong constant-bound-enumeration. For example, this fact is true of
homogeneous trees.
\begin{definition}
A tree $T$ is homogeneous iff $(\forall \rho_1,\rho_2\in T\ \forall \rho)(\rho/\rho_1\in T \Leftrightarrow \rho/\rho_2\in T)$.

\end{definition}

\begin{lemma}
\label{lem2}
Let $T\subseteq\omega^{<\omega}$ be a finitely branching homogeneous
tree, and let $\mathbf{a}$ be a Turing degree such that $[T]$ is
$\mathbf{a}$-effectively closed. Then $\mathbf{a}$ computes a
strong $k$-enumeration of $T$ for some
$k\in\mathbb{N}$ iff $\mathbf{a}$ computes a path through $T$.
\end{lemma}

\begin{proof}
$(\Leftarrow)$:
Straightforward, since an enumeration of the initial segments of a
path through $T$ is a strong 1-enumeration of $T$.

$(\Rightarrow)$:
We prove this by induction. Let $h\leq_T \mathbf{a}$ be a
strong $k$-enumeration of $T$. Clearly the
conclusion follows when $k=1$.

Suppose $\exists N\in\mathbb{N}\forall n>N\ D_{h(n)}\subset T$. Then
clearly one can compute a path through $T$ as follows: Let $g(n)$ be
the first integer $g_n$ we find such that $\exists m>N \exists
\sigma\in D_{h(m)},\ \sigma(n)=g_n$. Since $T$ is homogeneous, $g$ is a
path through $T$.

Suppose $\forall N\in\mathbb{N}\exists n>N\ D_{h(n)}\not\subset
T$. Since $[T]$ is $\mathbf{a}$-effectively closed, which means
$\bar{T}$ is $\mathbf{a}$-c.e., $\{n\in\mathbb{N}: D_{h(n)}\not\subset
T\}$ is an infinite $\mathbf{a}$-c.e.\ set. Therefore one can compute
a strong $(k-1)$-enumeration of $T$ as follows:
Search for an $m$ and a $\sigma$ such that $\sigma\in D_{h(m)}\cap
\bar{T}_m$ and $\forall \rho\in D_{h(m)}\ |\rho|\geq n$, and let
$g(n)$ be the index of $\{\rho\upharpoonright_n:\rho\in D_{h(m)}-\{\sigma\}\}$.

Note that $|D_{g(n)}|\leq k-1$ and $D_{g(n)}\cap T\ne\emptyset$
(as otherwise $D_{h(m)}\cap T=\emptyset$), i.e.\ $g\in\omega^\omega$
is a strong $(k-1)$-enumeration  of $T$. The conclusion follows by induction.
\end{proof}

Combining Theorem \ref{th1} and Lemma \ref{lem2} yields:

\begin{corollary}
\label{cor2}
Let $\mathcal{M}$ be a countable class closed under disjoint union,
and let $Q\subseteq \omega^{<\omega}$ be
an infinite $\mathcal{M}$-computably bounded homogeneous tree such
that $[Q]$ is $\mathcal{M}$-effectively closed. Then $$(\forall A\in
2^\omega)(\exists G \textrm{ s.t.\ }(G\subseteq A\vee G\subseteq
\bar{A}) \wedge |G|=\infty)(\forall C\in\mathcal{M},f\in[Q])\ f\nleq_T G\oplus C$$
if and only if $[Q]\nleq_u\mathcal{M}$. Here $\leq_u$ denotes the Muchnik reducibility.
\end{corollary}

Note that the computable boundedness is needed to ensure that in the proof of Lemma
\ref{fac8} the class $\mathcal{W}_n$ is $\Pi_1^{0,\mathcal{M}}$ uniformly in $n$,
which makes sure that the enumeration algorithm is computable in $\mathcal{M}$.

For $f\in\mathbb{N}^\mathbb{N}$, $\mathsf{f\textrm{-}WKL_0}$ as a second order
arithmetic statement says: for any set $X$ there exists a $g\in
f\textit{-DNR}^X$. $\mathsf{WKL_0}$ is $\mathsf{f\textrm{-}WKL_0}$ restricted
to $f\equiv 2$. Using Corollary \ref{cor2}, we have:
\begin{corollary}
\label{cor3}
For any computable $f\in\mathbb{N}^\mathbb{N}$ such that $f(n)>1$  i.o., $\mathsf{RT_2^2}$ does not imply $\mathsf{f\textrm{-}WKL_0}$.
\end{corollary}
\begin{proof}
Exactly the same as Corollary \ref{cor1-2}.
\end{proof}

We now mention an interesting result derived by Kjos-Hanssen
in algorithmic randomness theory.
\begin{corollary}[Kjos-Hanssen \cite{kjos2009infinite}]
\label{cor6}
For every Martin-L\"{o}f random set $A\in 2^\omega$, there exists an infinite
set $G\subseteq A\vee G\subseteq \bar{A}$ such that $G$ does not compute any Martin-L\"{o}f random set.
\end{corollary}
\begin{proof}
We've already shown in Lemma \ref{lem1} that every
finite joint union of the countably many
effectively closed sets
$[T_c]=\{X\in 2^{\omega}:\forall n,X\upharpoonright n \text{ is c-incompressible}_U\}$
satisfies the condition of Theorem \ref{th1}.
So the conclusion follows from Theorem \ref{th1} directly.
\end{proof}

\section{Remarks}
\label{sec6}
\begin{remark}
Although at each step of the construction we hold an
$\mathcal{M}$-computable tree $T$, updated by the R-ii-operation
and the P-operation, this is not essential. One could
pre-choose a path through $T$. Nonetheless, the construction tree
$\mathcal{T}$ is not avoidable. Furthermore, holding the tree requires
less computational power of the constructor.
\end{remark}
\begin{remark}
Can we cone avoid an \emph{arbitrary} tree $Q$ if $Q$ is
not Muchnik reducible to $\mathcal{M}$?
We guess not, and a construction showing that the answer
is no might be provided by modifying the proof of Theorem 2.3 of
\cite{hirschfeldt2008strength}, and noting that $\textit{Part}_k$
reduces a $k$-enumeration to a $1$-enumeration.
\end{remark}
\begin{remark}
We have modularized the construction into as many parts as we could,
which makes it convenient for further applications and
generalizations. For example, by modifying the R-ii-Operation, and
analyzing the proof of Lemma \ref{fac7}, one might strengthen Corollary
\ref{cor6} to prove Kjos-Hanssen's result in \cite{kjos2010strong}
that any Martin-L\"{o}f random set contains an infinite subset that does
not compute any Martin-L\"{o}f random set.

Another kind of application is to let $\textit{Part}_k$ serve as
a mediate as in Corollaries \ref{cor5} and \ref{cor4}, i.e.\ first
one proves that some class $\mathbb{X}$ is Muchnik reducible to
$\textit{Part}_k$ then it follows that some member of $\mathbb{X}$
does not compute any member of another class.
\end{remark}

\begin{remark}
A not so trivial generalization is to reinterpret $\subseteq,\cap$ and
therefore overload $\textit{Cross}$ and the definition of
``abandon''. This is explained as follows.

Say that we can strongly compute a class $\mathcal{Q}$ under the
condition $\langle\mathcal{B},\mathcal{C}(A)\rangle$, where
$\mathcal{B},\mathcal{C}(A)$ are classes, possibly with other
``parameters'', iff $\exists A\in\mathcal{B}\ \ \mathcal{Q}\leq_u
\mathcal{C}(A)$. (In this paper $\mathcal{B}=\mathcal{P}(\omega)$,
$\mathcal{C}(A)=\textit{Part}_2(A)$, $\mathcal{Q}$ is a closed set
with $Q$ non-$k$-enumerable in $\mathcal{M}$, and we could not
strongly compute $\mathcal{Q}$ under the condition
$\langle\mathcal{B},\mathcal{C}(A)\rangle$ for any $A$.) Therefore
$G\subseteq A$ is generalized as $G\in \mathcal{C}(A)$ and $X\cap Y$
is generalized as $\mathcal{C}(X)\cap \mathcal{C}(Y)$.
Note also that in many applications, $\mathcal{B}, \mathcal{C}(A)$ are
$\Pi_1^0$ classes, for example $\mathcal{B}=\{T\subseteq
2^{<\omega}:T\text{ is a tree satisfying certain property}\}$, say the
property of being $n$-bushy, and $\mathcal{C}(T)=[T]$ (or almost
$\Pi_1^0$ as $\textit{Part}_2(A)$, which is a $\Pi_1^0$ class minus a
countable class of sets). Therefore by redefining ``abandon'' and
$\textit{Cross}$, the main frame of the construction may still serve
for this purpose. (Of course, different combinatorial facts will be
needed in different cases and there would be no universal way to
produce these combinatorial facts; that belongs to the creative part
of mathematics. However, there might be some results with broad application,
since what we have studied is only a stone beside the sea.) We hope
that this frame could help in the study of homogeneous trees,
especially $f\textit{-DNR}$, and give universal solutions to various
results that have arisen recently, for example in
\cite{greenbergdiagonally,millerextracting}; see also
\cite{kumabe2009fixed}.

\end{remark}
\begin{question}
The most important question at hand is what kind of $\Pi_3^0$ class
can be avoided as above. Note that this is related to the question of
whether $\mathsf{SRT}_2^2$ implies $\mathsf{RT_2^2}$ which is
equivalent to whether $\mathsf{SRT_2^2}$ implies $\mathsf{COH}$.
Here $\mathsf{COH}$ is the second order arithmetic
statement that every sequence of sets $R^{(i)}$ has a cohesive set for
it, i.e.\ an infinite set $Y$ such that for all $i\in\mathbb{N}$ either $Y\cap
R^{(i)}$ or $Y\cap \bar{R}^{(i)}$ is finite. Note that the class of
cohesive sets of a given $R^{(i)}$ is a $\Pi_3^{0,R^{(i)}}$ class.

\end{question}

\begin{question}
What is the difference between the Muchnik lattice of $\textit{Part}_2$ and that of $\textit{Part}_3,\textit{Part}_k$?
(Here the Muchnik lattice of a collection of subsets of $\omega^\omega$ refers to the lattice where the $\leq$ relation is
induced by the Muchnik reducibility.)

For example, is there a $\textit{Part}_3(A_1,A_2)$ such that there exists no $\textit{Part}_2(A)$ to which $\textit{Part}_3(A_1,A_2)$ is Muchnik reducible? For the latter question we guess the answer is yes.
\end{question}

\bibliographystyle{amsplain}

\bibliography{cone}

\end{document}